\documentclass[1p,11pt]{elsarticle}
\usepackage[latin1]{inputenc}
\usepackage[T1]{fontenc}
\usepackage{amsmath}
\usepackage{amssymb}
\usepackage{amsfonts}
\usepackage{amstext}
\usepackage{amsthm}
\usepackage{enumerate}
\usepackage{graphics}
\usepackage{epsfig}
\usepackage{sidecap}
\usepackage{color}
\usepackage{xcolor}
\usepackage[mathscr]{euscript}
\usepackage{times}
\DeclareMathOperator{\codim}{codim}
\DeclareMathOperator{\dist}{dist}
\DeclareMathOperator*{\esssup}{ess\,sup}

\DeclareMathOperator{\sgn}{sgn}

\newcommand{\I}{{\mathbb I}}
\newcommand{\N}{{\mathbb N}}
\newcommand{\R}{{\mathbb R}}
\newcommand{\Z}{{\mathbb Z}}
\newcommand{\Q}{{\mathbb Q}}

\newcommand{\cC}{{\mathcal C}}

\newcommand{\cS}{{\mathcal S}}
\newcommand{\cT}{{\mathcal T}}
\newcommand{\LC}{{{LC}}}
\newcommand{\SC}{{{SC}}}
\newcommand{\sB}{{\mathscr B}}
\newcommand{\sF}{{\mathscr F}}
\newcommand{\sG}{{\mathscr G}}
\newcommand{\sH}{{\mathscr H}}
\newcommand{\sK}{{\mathscr K}}
\newcommand{\fB}{{\mathfrak B}}
\newcommand{\fC}{{\mathfrak C}}
\renewcommand{\d}{\,{\mathrm d}}
\newcommand{\tm}{\times}
\newcommand{\eps}{\varepsilon}
\newcommand{\intoo}[1]{\left(#1\right)}		% open interval
\newcommand{\intcc}[1]{\left[#1\right]}		% compact interval
\newcommand{\set}[1]{\left\{#1\right\}}		% set
\newcommand{\abs}[1]{\left|#1\right|}		% absolute value
\newcommand{\norm}[1]{\left\|#1\right\|}		% norm
\newcommand{\fall}{\;\text{ for all }}		% for all
\newcommand{\on}{\;\text{ on }}		% on
\newtheorem{theorem}{Theorem}[section]
\newtheorem{lemma}[theorem]{Lemma}
\newtheorem{corollary}[theorem]{Corollary}
\newtheorem{proposition}[theorem]{Proposition}

\theoremstyle{definition}

\newtheorem{example}[theorem]{Example}
\newtheorem*{hypothesis*}{Hypothesis}
\theoremstyle{remark}
\newtheorem{remark}[theorem]{Remark}
\newcommand{\cref}[1]{Cor.~\ref{#1}}
\newcommand{\eref}[1]{Ex.~\ref{#1}}
\newcommand{\fref}[1]{Fig.~\ref{#1}}
\newcommand{\href}[1]{Hyp.~\ref{#1}}
\newcommand{\pref}[1]{Prop.~\ref{#1}}
\newcommand{\tref}[1]{Thm.~\ref{#1}}
\newcommand{\lref}[1]{Lemma~\ref{#1}}

\allowdisplaybreaks
\author[1]{Iacopo P.~Longo}
\ead{iacopo.longo@imperial.ac.uk}
\author[2]{Christian P\"otzsche}
\ead{christian.poetzsche@aau.at}
\author[3]{Robert Skiba}
\ead{robert.skiba@mat.umk.pl}
\affiliation[1]{organization={Imperial College London, Department of Mathematics, DynamIC}, addressline={635 Huxley Building, 180 Queen's Gate South Kensington Campus}, 
city={London, SW7 2AZ}, 
country={United Kingdom}
}
\affiliation[2]{organization={Department of Mathematics, University of Klagenfurt},
addressline={Universit\"atsstra{\ss}e 65--67},
city={9020 Klagenfurt},
country={Austria}}
\affiliation[3]{organization={Faculty of Mathematics and Computer Science, Nicolaus Copernicus University in Toru\'n},
addressline={ul.\ Chopina 12/18},
city={87-100 Toru{\'n}},
country={Poland}}
\title{Global bifurcation of homoclinic solutions}
\begin{document}
\begin{abstract}
	In the analysis of parametrized nonautonomous evolutionary equations, bounded entire solutions are natural candidates for bifurcating objects. Appropriate explicit and sufficient conditions for such branchings, however, require to combine contemporary functional analytical methods from the abstract bifurcation theory for Fredholm operators with tools originating in dynamical systems. 

	This paper establishes alternatives classifying the shape of global bifurcating branches of bounded entire solutions to Carath{\'e}odory differential equations. Our approach is based on the parity associated to a path of index $0$ Fredholm operators, the global Evans function as a recent tool in nonautonomous bifurcation theory and suitable topologies on spaces of Carath{\'e}odory functions.
\end{abstract}
\maketitle
\section{Introduction}
Bifurcation theory is a central area in both abstract Nonlinear Functional Analysis \cite{kielhoefer:12, zeidler:95}, as well as in the field of Dynamical Systems \cite{kuznetsov:04}. A particularly symbiotic relation between these respective subdisciplines is reached in the recent Nonautonomous Bifurcation Theory \cite{anagnostopoulou:poetzsche:rasmussen:22}, which describes qualitative changes in the behavior of dynamical systems stimulated by aperiodic temporal influences. In contrast to the classical autonomous theory \cite{kuznetsov:04}, such time-driven dynamical systems typically do not have constant solutions and thus the set of potentially bifurcating objects needs to be extended. Indeed, rather than equilibria, for instance bounded entire or homoclinic solutions turn out to be suitable candidates. Yet, their analysis is heavily based on functional analytical tools and techniques \cite{anagnostopoulou:poetzsche:rasmussen:22, poetzsche:12}. 

The nonautonomous dynamical systems studied in this paper are differential equations of Carath{\'e}odory type \cite{bressan:piccoli:07,kurzweil:86}
\begin{equation}
	\tag{$C_\lambda$}
	\dot x=f(t,x,\lambda)
	\label{cde}
\end{equation}
in $\R^d$ depending on a real parameter $\lambda$. Such well-motivated problems are intrinsically time-variant, but extending their special case of classical ordinary differential equations, require merely measurable dependence on the time variable. Carath{\'e}odory equations typically arise as pathwise realizations of random differential equations \cite{arnold:98}, but are also omnipresent in mathematical control theory \cite{bressan:piccoli:07}. 

Local bifurcations of bounded entire solutions to \eqref{cde} were first studied in \cite{poetzsche:12}. For this endeavor we characterized these solutions of \eqref{cde} as zeros of a nonlinear differentiable operator $G$ depending on $\lambda$, i.e.\ by means of an abstract equation
\begin{equation}
	\tag{$O_\lambda$}
	G(x,\lambda)=0
	\label{abs}
\end{equation}
between ambient spaces of (essentially) bounded functions. The assumptions in \cite{poetzsche:12} were purely local in a bounded entire solution $\phi^\ast$ to $(C_{\lambda^\ast})$ for some critical parameter value $\lambda^\ast\in\Lambda$. They led to a detailed description of the bifurcation diagram for \eqref{cde} in the vicinity of a pair $(\phi^\ast,\lambda^\ast)$ in terms of a fold resp.\ a transversal intersection of two smooth branches of bounded entire solutions to \eqref{cde}. These results were a consequence of local bifurcation criteria applied to the abstract problem \eqref{abs} under appropriate conditions on \eqref{cde} guaranteeing that the linearization $D_1G(\phi^\ast,\lambda^\ast)$ is a Fredholm operator of index $0$ with $1$-dimensional kernel plus further conditions on the higher order partial derivatives of $G$ in $(\phi^\ast,\lambda^\ast)$. 

Our present alternative to the rather specific setting of \cite{poetzsche:12} provides the global assumption that a critical pair $(\phi^\ast,\lambda^\ast)$ is embedded into a prescribed continuous branch of bounded entire solutions $\phi_\lambda$ to \eqref{cde} such that the Fr{\'e}chet derivatives $D_1G(\phi_\lambda,\lambda)$ are Fredholm of index $0$ for parameters $\lambda$, but $D_1G(\phi^\ast,\lambda^\ast)$ having arbitrary kernel dimension. On the one hand, this is the framework to employ a topological invariant known as the \emph{parity} to \eqref{abs}, which was developed in \cite{fitzPeja86, FiPejsachowiczV, FitPejRab, fitzpatrick:etal:92, pejsachowicz:rabier:98} and received an axiomatization in \cite{lopez:sampedro:22}. On the other hand, an actual computation of the parity is involved and problem specific. In order to address this issue, we recently represented the parity of operators $G$ arising in the framework of Carath{\'e}odory equations \eqref{cde} using a concept similar to the \emph{Evans function}. In the stability theory of traveling waves to evolutionary PDEs this complex-analytical function is an established tool to study the (point) spectrum of their linearizations along traveling waves (cf.~\cite{kapitula:promislow:13, latushkin:pogan:15, sandstede:02}). Contrasting this, for our purposes an Evans function $E$ is real-valued and only needs to be continuous. Sign changes of $E$, however, turn out to be sufficient for local bifurcations of bounded entire solutions to \eqref{cde} (see~\cite[Thm.~4.2]{poetzsche:skiba:24}). Note also that in comparison to \cite{poetzsche:12} it suffices to assume that the partial derivative $(x,\lambda)\mapsto D_1G(x,\lambda)$ exists as continuous function.

The paper at hand now aims to provide information on the global structure of continua of bounded solutions branching off from $(\phi^\ast,\lambda^\ast)$. This requires to establish a global variant of the Evans function from \cite{poetzsche:skiba:24}. Then akin to the classical Rabinowitz alternative \cite{rabinowitz:71}, we obtain that a bifurcating continuum of bounded solutions to \eqref{cde} either returns to a prescribed branch, or it fails to be compact. Under further assumptions on the Carath{\'e}odory equation \eqref{cde} (and the prescribed solution branch) one can even establish that the bifurcating continuum is unbounded. We achieve this using abstract global bifurcation criteria. However, while the celebrated and widely used Rabinowitz alternative \cite{rabinowitz:71} relies on the Leray--Schauder degree, see e.g.~\cite[pp.~199ff]{kielhoefer:12}, our setting requires contemporary results from \cite{fitzpatrick:etal:92, galdi:rabier:99, pejsachowicz:rabier:98} using the more general and flexible parity developed in \cite{fitzPeja86, FiPejsachowiczV, FitPejRab}. A corresponding application to Carath{\'e}odory equations \eqref{cde} is essentially based on two key ingredients: 
\begin{itemize}
	\item There is a close relation between the partial derivatives $D_1G(x,\lambda)$ being Fredholm and the fact that the variational equations of \eqref{cde} along the bounded solutions $\phi_\lambda$ possess exponential dichotomies \cite{coppel:78} on both halflines. For linear ordinary differential equations this was established in \cite{palmer:84,palmer:88}, while the minor modifications necessary for Carath{\'e}odory equations are due to \cite{poetzsche:12,poetzsche:skiba:24}. Beyond this, the existence of appropriate exponential dichotomies allows to construct a global Evans function, whose sign changes yield local bifurcations (cf.~\cite{poetzsche:skiba:24}). 

	\item In order to obtain more detailed information on the global structure of the bifurcating continua using \cite{fitzpatrick:etal:92, galdi:rabier:99, pejsachowicz:rabier:98}, properness assumptions on $G$ are due. On the one hand, they are based on suitable compactness criteria in the spaces of essentially bounded and bounded weakly differentiable functions. On the other hand, there is a crucial connection between the topological notion of properness for $G$ and Topological Dynamics \cite{sell:71}. On subspaces of the bounded continuous functions this was first observed in \cite{rabier:04}, whereas our extensions to essentially bounded functions require nontrivial methods developed in \cite{thesis:Longo}. 
\end{itemize}
We eventually point out that our analysis reveals that the continua bifurcating from the pair $(\phi^\ast,\lambda^\ast)$ have a rather special structure. They consist of bounded solutions to \eqref{cde} being in fact perturbations of the prescribed branch. In fact they converge to a solution $\phi_\lambda$ in both time directions --- one speaks of solutions being \emph{homoclinic} to $\phi_\lambda$. 

The paper is structured as follows: Sect.~\ref{sec2} introduces a general class of parametrized Cara\-th{\'e}o\-dory equations \eqref{cde} and a continuously Fr{\'e}chet differentiable operator $G$ enabling us to characterize the bounded solutions of \eqref{cde} as zeros of $G$. Properness of the operator $G$ is derived in Sect.~\ref{sec3} and requires two preparations: First, we characterize the compactness of subsets of the spaces $L_0^\infty(\R)$ and $W_0^{1,\infty}(\R)$ consisting of essentially bounded functions vanishing at $\pm\infty$. Second, we investigate the Bebutov flow over the hull of \eqref{cde}, which necessitates fairly novel results ensuring compactness of the set of time-translated Carath{\'e}odory functions $f(\cdot,\lambda)$ (see \cite{thesis:Longo}). Then in terms of appropriate one-sided exponential dichotomies for the variational equation associated to \eqref{cde}, our Sect.~\ref{sec4} establishes that the globally defined $G$ is a nonlinear Fredholm operator (of index $0$) and that the Evans function can be extended continuously to the entire parameter interval. After these preparations, our main results are obtained in Sect.~\ref{sec5}. Based on bifurcation points identified via sign changes in global Evans functions, \tref{thmbif} provides general alternatives for the global structure of the bifurcating continua. They consist of bounded entire solutions to \eqref{cde} being asymptotically equivalent to the members $\phi_\lambda$ of the prescribed branch. Appropriate admissibility assumptions (formulated in terms of the Bebutov flow) allow to deduce even unboundedness of the bifurcating continua in \tref{thmglobal2CDE}. For bounded continua, however, we obtain that the oriented sum over the sign changes of a global Evans function is always zeri (cf.~\tref{thmsignCDE}). This corresponds to a vanishing \emph{bifurcation index} used in the specifications \cite[p.~342, (1.9), p.~344, (1.10)]{granas:dugundji:03} of the classical alternative from \cite{rabinowitz:71} or \cite[p.~205, Thm.~II.3.3]{kielhoefer:12} when the Leray--Schauder degree applies. On this basis, sufficient conditions for bifurcating branches to be unbounded result in \cref{corsignCDE}. A simple Carath{\'e}odory equation illustrates the feasibility of these alternatives in the final Sect.~\ref{sec6}. For the convenience of the reader, three appendices provide the required basics on Topological Dynamics, present sufficient conditions for properness of parametrized operators $G$ and finally derive the fundamental bifurcation results Thms.~\ref{thmglobal}, \ref{thmglobal2} and \ref{thmsign} from predecessors in \cite{fitzpatrick:etal:92, galdi:rabier:99, pejsachowicz:rabier:98} and \cite{lopez:mora:04,lopez:sampedro:22,lopez:sampedro:24} based on the parity. 
\paragraph{Notation}
We write $\R_+:=[0,\infty)$ and $\R_-:=(-\infty,0]$ for the halflines of the reals $\R$. 

Let $X$ be a Banach space with norm $\norm{\cdot}$ or $\norm{\cdot}_X$. The \emph{distance} of a point $x\in X$ to a subset $A\subseteq X$ is abbreviated as $\dist_A(x):=\inf_{a\in A}\norm{x-a}$; we write $\partial A$, $A^\circ$ and $\overline{A}$ for its boundary, interior resp.\ its closure. Norms on $\R^d$ (and further finite-dimensional spaces) will be denoted by $\abs{\cdot}$. 

Given a further Banach space $Y$, we write $L(X,Y)$ for the Banach space of bounded linear operators $T:X\to Y$ having the bounded invertible operators $GL(X,Y)$ and the index $0$ Fredholm operators $F_0(X,Y)$ as open subsets in the operator topology. It is convenient to abbreviate $L(X):=L(X,X)$ and $GL(X):=GL(X,X)$. Furthermore, $R(T):=TX\subseteq Y$ denotes the range and $N(T):=T^{-1}(0)\subseteq X$ the kernel of $T$. 
\paragraph{Function spaces}
Our functional analytical approach requires a suitable spatial setting of functions defined on intervals $I\subseteq\R$ with values in nonempty, open subsets $\Omega\subseteq\R^d$. We write $BC(I,\Omega)$ for the set of bounded continuous functions $x:I\to\Omega$. It is convenient to abbreviate $BC(I):=BC(I,\R^d)$ and we proceed similarly with further function spaces. When equipped with the norm $\norm{x}_\infty:=\sup_{t\in I}\abs{x(t)}$ this becomes a real Banach space. On $I=\R$ we introduce the subspace $C_0(\R)$ of continuous functions satisfying $\lim_{t\to\pm\infty}x(t)=0$, as well as the subspace $C_0^1(\R)$ of continuously differentiable functions $x\in C_0(\R)$ such that $\dot x\in C_0(\R)$ holds, which is equipped with the norm $\norm{x}_{1,\infty}:=\max\set{\norm{x}_\infty,\norm{\dot x}_\infty}$. 

Throughout, measurability and integrability are understood in the Lebesgue sense. We write $L_{loc}^1(I)$ for the $\R^d$-valued functions being integrable on each compact subset of $I$. Moreover, $L^\infty(I,\Omega)$ denotes the set of essentially bounded functions $x:I\to\Omega$ and $W^{1,\infty}(I,\Omega)$ is the set of such $L^\infty$-functions with essentially bounded (weak) derivatives. We note that $L^\infty(I)$ is a real Banach space with the norm
$
	\norm{x}_\infty:=\esssup_{t\in I}\abs{x(t)}.
$

Every $x\in W^{1,\infty}(I)$ has a bounded Lipschitz continuous representative (cf.\ \cite[p.~224, Thm.~7.17]{leoni:09}). First, Rademacher's theorem \cite[p.~343, Thm.~11.49]{leoni:09} ensures that the (strong) derivative $\dot x:I\to\R^d$ exists a.e.\ in $I\subseteq\R$. Second, this implies that $x$ is absolutely continuous and the Fundamental Theorem of Calculus \cite[p.~85, Thm.~3.30]{leoni:09} applies. From \cite[p.~224, Ex.~7.18]{leoni:09} we see that $W^{1,\infty}(I)$ is a real Banach space with norm
$
	\norm{x}_{1,\infty}:=\max\set{\norm{x}_\infty,\norm{\dot x}_\infty}.
$
Clearly, $W^{1,\infty}(I)\subseteq L^\infty(I)$ is a continuous embedding. On $I=\R$ we furthermore introduce the closed subspaces\footnote{the limit is to be understood in the essential sense, i.e.\ for each $\eps>0$ there exists a real $T_\eps>0$ such that $\abs{x(t)}<\eps$ for a.a.~$t\in\R\setminus(-T_\eps,T_\eps)$.}
\begin{align*}
	L_0^\infty(\R)&:=\bigl\{x\in L^\infty(\R):\,\lim_{t\to\pm\infty}x(t)=0\bigr\},\\
	W_0^{1,\infty}(\R)&:=\bigl\{x\in W^{1,\infty}(\R):\,x,\dot x\in L_0^\infty(\R)\bigr\}
\end{align*}
and arrive at the continuous embeddings
\begin{equation}
	\begin{array}{ccccc}
		C_0(\R) & \hookrightarrow & L_0^\infty(\R) & \hookrightarrow & L^\infty(\R)\\
        \rotatebox[origin=c]{90}{$\hookrightarrow$} & & \rotatebox[origin=c]{90}{$\hookrightarrow$} & & \rotatebox[origin=c]{90}{$\hookrightarrow$}\\
        C_0^1(\R) & \hookrightarrow & W_0^{1,\infty}(\R) & \hookrightarrow & W^{1,\infty}(\R).
    \end{array}
    \label{noembed}
\end{equation}
\section{Carath{\'e}odory equations}
\label{sec2}
We are interested in Carath{\'e}odory equations
\begin{equation}
	\tag{$C_\lambda$}
	\dot x=f(t,x,\lambda)
\end{equation}
depending on a parameter $\lambda\in\Lambda$ from an open interval $\Lambda\subseteq\R$ under the assumptions: 
\begin{hypothesis*}[$\mathbf{H_0}$] Let $\Omega\subseteq\R^d$ be nonempty, open and convex. The right-hand side $f:\R\tm\Omega\tm\Lambda\to\R^d$ of \eqref{cde} is a Carath{\'e}odory function having the properties:
\begin{itemize}
    \item for each parameter $\lambda\in\Lambda$ the function $f(\cdot,\lambda):\R\tm\Omega\to\R^d$ is measurable and $f(t,\cdot,\lambda):\Omega\to\R^d$ is differentiable for a.a.~$t\in\R$ with measurable partial derivative $D_2f(t,\cdot):\Omega\tm\Lambda\to\R^{d\tm d}$, 

    \item for each parameter $\lambda\in\Lambda$ and compact $K\subset\Omega$ there exists a real $m_K(\lambda)\geq 0$ such that for a.a.~$t\in\R$,
\begin{equation}
	\abs{D_2^jf(t,x,\lambda)}\leq m_K(\lambda)\fall x\in K\text{ and }j=0,1, 
	\label{h0b}
\end{equation}

    \item for all $\eps>0$ there is a $\delta>0$ such that for a.a.\ $t\in\R$, 
\begin{equation}
	\abs{x-\bar x}<\delta,\,
	\abs{\lambda-\bar\lambda}<\delta
	\quad\Rightarrow\quad
	\abs{D^j_2f(t,x,\lambda)-D^j_2f(t,\bar x,\bar\lambda)}<\eps
	\label{h0c}
\end{equation}
for all $x,\bar x\in\Omega$, parameters $\lambda,\bar\lambda\in\Lambda$ and $j=0,1$. 
\end{itemize}
\end{hypothesis*}

\begin{remark}\label{rmk:Carat-assump} 
    Keeping $\lambda\in\Lambda$ fixed, in the jargon of Carath\'eodory functions the conditions required in Hypothesis~$(H_0)$ are equivalent to asking that the right-hand side $f(\cdot,\lambda)$ is \emph{Lipschitz Carath\'eodory}, in short $f(\cdot,\lambda)\in\LC$, and differentiable in $x$ for a.a.~$t\in\R$ with partial derivative $D_2f$ continuous in $x$ for a.a.~$t\in\R$, i.e.~$D_2f(\cdot,\lambda)$ is \emph{Strong Carath\'eodory}, in short $D_2f(\cdot,\lambda)\in \SC$, and both $f(\cdot,\lambda)$ and $D_2f(\cdot,\lambda)$ are locally essentially bounded. We refer the interested reader to App.~\ref{appA} for further details.
\end{remark}

For fixed $\lambda\in\Lambda$, a \emph{solution} to a Carath{\'e}odory equation \eqref{cde} is a continuous function $\phi:I\to\Omega$ on an open interval $I\subseteq\R$ satisfying the Volterra integral equation (cf.~\cite{bressan:piccoli:07})
$$
    \phi(t)=\phi(\tau)+\int_\tau^tf(s,\phi(s),\lambda)\d s\fall\tau,t\in I.
$$
This implies that $\phi$ is absolutely continuous on every bounded subinterval $J\subseteq I$ and therefore strongly differentiable a.e.\ in $I$. In case $I=\R$, one speaks of an \emph{entire solution}. Note that due to Hypothesis~$(H_0)$ the integral above is well-defined because $s\mapsto f(s,\phi(s),\lambda)$ is locally integrable, and the existence and uniqueness of maximal solutions for Carath\'eodory initial value problems is guaranteed (see e.g.\ \cite{bressan:piccoli:07, kurzweil:86}).

\begin{hypothesis*}[$\mathbf{H_1}$] The Carath{\'e}odory equation \eqref{cde} has a continuous branch $(\phi_\lambda)_{\lambda\in\Lambda}$ of bounded entire solutions, i.e.\ each $\phi_\lambda:\R\to\Omega$, $\lambda\in\Lambda$, is a bounded entire solution to \eqref{cde} and for every $\eps>0$, $\lambda_0\in\Lambda$ there exists a $\delta>0$ such that
	\begin{equation*}
		\abs{\lambda-\lambda_0}<\delta
		\quad\Rightarrow\quad
		\sup_{t\in\R}\abs{\phi_\lambda(t)-\phi_{\lambda_0}(t)}<\eps\fall\lambda\in\Lambda.
	\end{equation*}
	Moreover, the branch $(\phi_\lambda)_{\lambda\in\Lambda}$ is \emph{permanent}, i.e.\
	\begin{equation}
		\inf_{\lambda\in\Lambda}\inf_{t\in\R}\dist_{\partial\Omega}(\phi_\lambda(t))>0.
		\label{nop}
	\end{equation}
\end{hypothesis*}
Given $\lambda\in\Lambda$, an entire solution $\phi:\R\to\Omega$ of \eqref{cde} is said to be \emph{homoclinic to} $\phi_\lambda$, if the limit relation $\lim_{t\to\pm\infty}\abs{\phi(t)-\phi_\lambda(t)}=0$ holds, but $\phi\neq\phi_\lambda$. 

Our next aim is to embed parametrized Carath{\'e}odory equations \eqref{cde} into a suitable functional analytical framework. Thereto, we characterize entire solutions of \eqref{cde} being homoclinic to $\phi_\lambda$ as zeros 
\begin{equation}
	\tag{$O_\lambda$}
	G(x,\lambda)=0
\end{equation}
of formally defined abstract nonlinear operators
\begin{equation}
	[G(x,\lambda)](t)
	:=
	\dot x(t)-f(t,x(t)+\phi_\lambda(t),\lambda)+f(t,\phi_\lambda(t),\lambda). 
	\label{gdef}
\end{equation}

\begin{lemma}[superposition operator]\label{lemfprop}
	If $(H_0$--$H_1)$ hold, then the \emph{superposition operator} $F:U_0\to L_0^\infty(\R)$ defined as
	$$
		[F(x,\lambda)](t):=f(t,x(t)+\phi_\lambda(t),\lambda)-f(t,\phi_\lambda(t),\lambda)\quad\text{for a.a.}~t\in\R
	$$
	is well-defined and continuous on the open set \[U_0:=\set{(x,\lambda)\in L_0^\infty(\R)\tm\Lambda:\,x(t)+\phi_\lambda(t)\in\Omega\text{ a.e.\ in }\R}.\]
 Moreover, its partial derivative 
    \begin{align*}
        D_1F:U_0&\to L(L_0^\infty(\R)),&
        [D_1F(x,\lambda)h](t)&=D_2f(t,x(t)+\phi_\lambda(t),\lambda)h(t)
    \end{align*}
    exists as a continuous function. 
\end{lemma}
\begin{proof}
	Let $(x_0,\lambda_0)\in U_0$. We initially indicate that $U_0$ is open. First, \eqref{nop} guarantees that there is a $\rho_0>0$ such that $x(t)+\phi_\lambda(t)\in\Omega$ a.e.\ in $\R$ holds for each $L_0^\infty$-function $x\in B_{\rho_0}(x_0)$. Second, since $\Lambda$ is open, there is a $B_{\rho_1}(\lambda_0)\subseteq\Lambda$ and we conclude $B_\rho(x_0,\lambda_0)\subseteq U_0$ with $\rho:=\min\set{\rho_0,\rho_1}$, i.e.\ $U_0$ is open. 
	
	Since $x_0$ is essentially bounded and $\phi_{\lambda_0}$ is bounded, and $\Omega$ is convex, there exists a compact $K\subseteq\Omega$ with $x_0(t),\phi_{\lambda_0}(t)+\theta x_0(t)\in K$ for a.a.\ $t\in\R$ and $\theta\in[0,1]$. It saves space to abbreviate $y_0:=x_0+\phi_{\lambda_0}$. 

	(I) By assumption \eqref{h0b} there is a $m_K\geq 0$ with $\abs{D_2f(t,\phi_{\lambda_0}(t)+\theta x_0(t),\lambda_0)}\leq m_K$ for a.a.\ $t\in\R$. With the Mean Value Theorem \cite[p.~243, Thm.~4.C for $n=1$]{zeidler:95} results
	\begin{align*}
        &
		\abs{f(t,y_0(t),\lambda_0)-f(t,\phi_{\lambda_0}(t),\lambda_0)}\\
		&\leq
        \int_0^1\abs{D_2f(t,\phi_{\lambda_0}(t)+\theta x_0(t),\lambda_0)}\d\theta\abs{x_0(t)}
        \stackrel{\eqref{h0b}}{\leq}
        m_K\abs{x_0(t)}
		\quad\text{for a.a.~}t\in\R
	\end{align*}
	and therefore $\lim_{t\to\pm\infty}x_0(t)=0$ implies the limit relation
    $$
        \lim_{t\to\pm\infty}\abs{f(t,y_0(t),\lambda_0)-f(t,\phi_{\lambda_0}(t),\lambda_0)}=0,
    $$
    i.e.\ $F(x_0,\lambda_0)\in L_0^\infty(\R)$ and $F$ is well-defined. 

	(II) Given $\eps>0$, according to $(H_0)$ there exists a $\delta_0>0$ so that for any $x\in B_{\delta_0}(x_0)$, $\lambda\in B_{\delta_0}(\lambda_0)$ one has
    $\abs{f(t,x,\lambda)-f(t,x_0,\lambda_0)}<\tfrac{\eps}{3}$ for a.a.~$t\in\R$, while $(H_1)$ yields the existence of a $\delta_1>0$ such that $\lambda\in B_{\delta_1}(\lambda_0)$ implies $\abs{\phi_\lambda(t)-\phi_{\lambda_0}(t)}<\tfrac{\delta_0}{3}$ for all $t\in\R$ and for $x\in B_{\delta_0/3}(x_0)$ one obtains
    $$
        \abs{x(t)+\phi_\lambda(t)-y_0(t)}
        \leq
        \abs{x(t)-x_0(t)}+\abs{\phi_\lambda(t)-\phi_{\lambda_0}(t)}
        <
        \delta_0
    $$
    for a.a.\ $t\in\R$. In conclusion, this implies 
	\begin{align*}
        &
		\abs{f(t,x(t)+\phi_\lambda(t),\lambda)-f(t,\phi_\lambda(t),\lambda)-f(t,y_0(t),\lambda_0)+f(t,\phi_{\lambda_0}(t),\lambda_0)}\\
        &\leq
        \abs{f(t,x(t)+\phi_\lambda(t),\lambda)-
        f(t,y_0(t),\lambda_0)}+\abs{f(t,\phi_\lambda(t),\lambda)-f(t,\phi_{\lambda_0}(t),\lambda_0)}
        <
        \tfrac{2}{3}\eps
	\end{align*}
	for a.a.~$t\in\R$ and passing to the essential supremum over $t\in\R$ yields the estimate $\norm{F(x,\lambda)-F(x_0,\lambda_0)}_\infty<\eps$ for all $x\in B_\delta(x_0)$, $\lambda\in B_\delta(\lambda_0)$ with $\delta:=\min\set{\delta_0,\delta_1}$. This establishes the continuity of $F$ in an arbitrary pair $(x_0,\lambda_0)$. 

	(III) We pointwise define the linear multiplication operator
	$$
		[Mh](t):=D_2f(t,y_0(t),\lambda_0)h(t)
		\quad\text{for a.a.~}t\in\R
	$$
	and $h\in L_0^\infty(\R)$. Passing to the essential supremum in the estimate
	$$
		\abs{D_2f(t,y_0(t),\lambda_0)h(t)}
		\stackrel{\eqref{h0b}}{\leq}
		m_K\abs{h(t)}
		\leq
		m_K\norm{h}_\infty
	$$
	for a.a.~$t\in\R$ leads to $\norm{Mh}_\infty\leq m_K\norm{h}_\infty$ and $M:L_0^\infty(\R)\to L_0^\infty(\R)$ is a bounded linear operator. Since the Mean Value Theorem \cite[p.~243, Thm.~4.C for $n=1$]{zeidler:95} implies
	\begin{align*}
		&
		|f(t,x_0(t)+h(t)+\phi_{\lambda_0}(t),\lambda_0)-
  f(t,\phi_{\lambda_0}(t),\lambda_0)\\
        &
        \quad
  -f(t,y_0(t),\lambda_0)+f(t,\phi_{\lambda_0}(t),\lambda_0)-D_2f(t,y_0(t),\lambda_0)h(t)|\\
		&=
		|f(t,x_0(t)+h(t)+\phi_{\lambda_0}(t),\lambda_0)-f(t,y_0(t),\lambda_0)\\
        &
        \quad-D_2f(t,y_0(t),\lambda_0)h(t)|\\
		&\leq
		\int_0^1\abs{D_2f(t,y_0(t)+\theta h(t),\lambda_0)-D_2f(t,y_0(t),\lambda_0)}\d\theta\abs{h(t)}
		\leq
		r(h)\norm{h}_\infty
	\end{align*}
	for a.a.~$t\in\R$ with the real-valued remainder function
	$$
		r(h):=\esssup_{t\in\R}\sup_{\theta\in[0,1]}
		\abs{D_2f(t,y_0(t)+\theta h(t),\lambda_0)-D_2f(t,y_0(t),\lambda_0)},
	$$
	yields 
	$
		\norm{F(x_0+h,\lambda_0)-F(x_0,\lambda_0)-Mh}_\infty
		\leq
		r(h)\norm{h}_\infty.
	$
	On the one hand, the uniform continuity assumption \eqref{h0c} implies $\lim_{h\to0}r(h)=0$ and therefore $F(\cdot,\lambda_0)$ is differentiable in $x_0$ with the derivative $D_1F(x_0,\lambda_0)h=Mh$. On the other hand, it follows as in step (II) using \eqref{h0c} that $D_1F:U_0\to L(L_0^\infty(\R))$ is a continuous function. 
\end{proof}

\begin{theorem}[properties of $G$]\label{thmgprop}
	If $(H_0$--$H_1)$ hold, then the operator
	$
		G:U\to L_0^\infty(\R)
	$
	given by \eqref{gdef} has the following properties:
	\begin{enumerate}[(a)]
		\item Its domain $U:=\{(x,\lambda)\in W_0^{1,\infty}(\R)\tm\Lambda:\,x(t)+\phi_\lambda(t)\in\Omega\text{ for all }t\in\R\}$ is nonempty, open and simply connected, 

		\item $G$ is well-defined and continuous with $G(0,\lambda)\equiv 0$ on $\Lambda$, 

		\item the partial derivative $D_1G:U\to L(W_0^{1,\infty}(\R),L_0^\infty(\R))$, 
		\begin{equation}
			[D_1G(x,\lambda)y](t)=\dot y(t)-D_2f(t,x(t)+\phi_{\lambda}(t),\lambda)y(t)
			\label{thmgprop1}
		\end{equation}
		exists as a continuous function. 
    \end{enumerate}
\end{theorem}
\begin{proof}
	(a) From the convexity of $\Omega$ we deduce that $U$ is simply connected. The argument for the openness of $U_0\subseteq L_0^\infty(\R)\tm\Lambda$ in the proof of \lref{lemfprop} carries over to the subset $U\subseteq W_0^{1,\infty}(\R)\tm\Lambda$; one has $U\subseteq U_0$. 

	(b) and (c) The continuous embeddings \eqref{noembed} have two consequences: First, $x\mapsto\dot x$ is a linear bounded mapping $W_0^{1,\infty}(\R)\to L_0^\infty(\R)$. Second also the restriction $F|_U$ of the superposition operator from \lref{lemfprop} is well-defined and continuous with continuous partial derivative $D_1F|_U$. Since $G$ can be represented as $G(x,\lambda):=\dot x-F(x,\lambda)$, this yields the claims. 
\end{proof}

\begin{theorem}\label{thmchar}
	If $(H_0$--$H_1)$ hold, then the following is true for all $\lambda\in\Lambda$: 
	\begin{enumerate}[$(a)$]
		\item If $\phi:\R\to\Omega$ is a solution of \eqref{cde} homoclinic to $\phi_\lambda$, then the difference $\phi-\phi_\lambda$ is contained in $W_0^{1,\infty}(\R)$ and satisfies \eqref{abs},

		\item if $\psi\in L_0^\infty(\R)$ has a (strong) derivative a.e.\ in $\R$ and satisfies $G(\psi,\lambda)=0$, then $\psi\in W_0^{1,\infty}(\R)$ and $\psi+\phi_\lambda:\R\to\Omega$ is a solution of \eqref{cde} homoclinic to $\phi_\lambda$.
	\end{enumerate}
\end{theorem}
\begin{proof}
    Let $\lambda\in\Lambda$ be fixed.

    (a) If $\phi:\R\to\Omega$ is an entire solution of \eqref{cde} homoclinic to $\phi_\lambda$, then the difference $\delta:=\phi-\phi_\lambda\in L_0^\infty(\R)$ satisfies $\dot\delta(t)+\dot\phi_\lambda(t)=f(t,\delta(t)+\phi_\lambda(t),\lambda)$ and consequently
    $$
        \dot\delta(t)
        =
        f(t,\delta(t)+\phi_\lambda(t),\lambda)-f(t,\phi_\lambda(t),\lambda)
        \quad\text{for a.a.\ }t\in\R.
    $$
    This has two consequences: First, there exists a compact $K\subseteq\Omega$ such that the inclusion $\phi_\lambda(t)+\theta\delta(t)\in K$ holds for all $t\in\R$, $\theta\in[0,1]$ and whence again the Mean Value Theorem \cite[p.~243, Thm.~4.C for $n=1$]{zeidler:95} implies
    \begin{align*}
        \abs{\dot\delta(t)}
        &=
        \abs{\int_0^1D_2f(t,\theta\delta(t)+\phi_\lambda(t),\lambda)\d\theta\delta(t)}\\
        &\leq
        \int_0^1\abs{D_2f(t,\theta\delta(t)+\phi_\lambda(t),\lambda)\d\theta}\abs{\delta(t)}
        \stackrel{\eqref{h0b}}{\leq}
        m_K\abs{\delta(t)}\xrightarrow[t\to\pm\infty]{}0,
    \end{align*}
    i.e.\ $\delta\in W_0^{1,\infty}(\R)$ holds. Second, $\delta$ defines an entire solution of the equation of perturbed motion $\dot x=f(t,x+\phi_\lambda(t),\lambda)-f(t,\phi_\lambda(t),\lambda)$, which in turn implies $G(\delta,\lambda)=0$. 
 
    (b) Let $\psi\in L_0^\infty(\R)$ satisfy $G(\psi,\lambda)=0$, i.e.\ 
    $$
		\dot\psi(t)
        \stackrel{\eqref{gdef}}{=}
        f(t,\psi(t)+\phi_\lambda(t),\lambda)-f(t,\phi_\lambda(t),\lambda)
        \quad\text{for a.a.\ }t\in\R.
    $$
    This means, $\dot\psi(t)+\dot\phi_\lambda(t)\equiv f(t,\psi(t)+\phi_\lambda(t),\lambda)$ a.e.\ on $\R$ implies that $\psi+\phi_\lambda$ is a solution of \eqref{cde} homoclinic to $\phi_\lambda$. Moreover, as in (a) one establishes $\dot\psi\in L_0^\infty(\R)$ and therefore $\psi\in W_0^{1,\infty}(\R)$ holds. 
\end{proof}
\section{Properness}
\label{sec3}
The above \tref{thmgprop} characterizes certain bounded entire solutions to a Carath{\'e}odory equation \eqref{cde} as zeros of an abstract equation \eqref{abs}. In order to apply appropriate tools from Functional Analysis to solve \eqref{abs}, however, the operator $G:U\to L_0^\infty(\R)$ needs to satisfy further conditions. The first of them is a weakened version of \emph{properness} (see App.~\ref{appB}), i.e.\ the fact that $G$-preimages of compact sets are compact again. For this purpose one needs a detailed understanding which subsets of the function spaces $L_0^\infty(\Omega)$ and $W_0^{1,\infty}(\R)$ are compact. Thereto, for a function $x:\R\to\R^d$ we define the \emph{shift} $(S^tx)(s):=x(t+s)$ for all $s,t\in\R$.

%\comment{IPL: Maybe we need to write a couple of lines here on why we tackle %the question of properness. In fact, we title this section "Properness" and %give sufficient results for properness at the end but the definition of %properness only comes in the appendix. It might be of help to give some %clarifications. What do you think?\\
%CP: Better now?\\
%IPL: I am very happy with this.}
Given a closed and totally disconnected set $Z\subset\R^d$ let us introduce the set
$$
    C_Z:=\set{x\in BC(\R)\,\bigg|\,\lim_{t\to\pm\infty}\dist_Z(x(t))=0}
$$
of all bounded continuous functions converging to $Z$ in both time directions. We borrow the following two compactness criteria from the literature:
%\begin{lemma}[compactness in $C_Z(\R)$, {cf.~\cite[Cor.~6]{rabier:04}}]
%    \label{lem:comp_in_C_Z}
%    A subset $\sF\subset C_Z(\R)$ is relatively compact, if and only if the following holds: 
%\begin{enumerate}[$(i)$]
%	\item $\sF$ is bounded,
%
%    \item $\sF$ is \emph{uniformly equicontinuous}, that is, for every $\eps>0$ there exists a $\delta>0$ such that %$\abs{t-s}<\delta$ implies $\abs{x(t)-x(s)}<\eps$ for all $x\in\sF$ and $t,s\in\R$,
%
%    \item if $(x_n)_{n\in\N}$ is a sequence in $\sF$, and  $(t_n)_{n\in\N}$ is a sequence in $\R$ with %$\lim\limits_{n\to\infty}|t_n|=\infty$ and  $\lim\limits_{n\to\infty}(S^{t_n}x_n)(t)=\bar x(t)$ for all $t\in\R$, then $\bar x\in %C_Z$ and $\bar x(\R)\subset Z$.
%\end{enumerate}
%In particular, if $Z\subset\R^d$ is compact, then for all $z\in Z$, $C_{\{z\}}\subset C_Z$ and $\sF\subset C_{\{z\}}$ is %relatively compact in $C_{\{z\}}$ if and only if $\set{z}$ is relatively compact in $C_Z$.
%\end{lemma}
\begin{lemma}[compactness in $C_Z$, {cf.~\cite[Cors.~6, 7]{rabier:04}}]
    \label{lem:comp_in_C_Z}
    A subset $\sF\subseteq C_Z$ is relatively compact, if and only if the following holds: 
    \begin{enumerate}[$(i)$]
		\item $\sF$ is bounded,

        \item $\sF$ is \emph{uniformly equicontinuous}, that is, for every $\eps>0$ there exists a $\delta>0$ such that $\abs{t-s}<\delta$ implies $\abs{x(t)-x(s)}<\eps$ for all $x\in\sF$ and $t,s\in\R$,

        \item if $(x_n)_{n\in\N}$ is a sequence in $\sF$ and $(t_n)_{n\in\N}$ is a sequence in $\R$ with $\lim\limits_{n\to\infty}|t_n|=\infty$ such that $\bar x(t):=\lim\limits_{n\to\infty}(S^{t_n}x_n)(t)$ for all $t\in\R$ defines a function $\bar x\in BC(\R)$, then $\bar x(\R)\subset Z$. 
        %\comment{CP: Modified such that it actually fits Rabier, Cor.\ 6 and %7}
    \end{enumerate}
    In particular, given any $z\in\R^d$, a subset $\sF\subset C_{\set{z}}$ is relatively compact if and only if (i), (ii) hold and there exists a compact and totally disconnected set $Z\subseteq\R^d$ such that $(iii)$ is fulfilled.
\end{lemma}

\begin{lemma}[compactness in $L^\infty(I)$, {cf.~\cite[Thm.~3.9]{eveson:00}}]
	\label{lemcomp}
	Let $I\subseteq\R$ be an interval. A subset $\sF\subseteq L^\infty(I)$ is relatively compact, if and only if the following holds:
	\begin{enumerate}[$(i)$]
		\item $\sF$ is bounded, 
		\item $\sF$ is \emph{uniformly equimeasurable} on $I$, that is, for every $\eps>0$ there exists a partition $\set{P_1,\ldots,P_n}$ of $I$ such that for each $j\in\set{1,\ldots,n}$ and for a.a.\ $s,t\in P_j$ one has $\abs{x(t)-x(s)}<\eps$ for all $x\in\sF$.
	\end{enumerate}
\end{lemma}

Given this, we arrive at the following characterization of compactness in the function spaces $L_0^\infty(\R)$ and $W_0^{1,\infty}(\R)$ immediately relevant for our purposes:
\begin{proposition}[compactness in $L_0^\infty(\R)$]
	\label{propcom1}
	A subset $\sF\subseteq L_0^\infty(\R)$ is relatively compact, if and only if the following holds:
	\begin{enumerate}[$(i)$]
		\item $\sF$ is bounded, 

		\item $\sF$ is uniformly equimeasurable on $\R$, 

		\item $\sF$ is \emph{uniformly essentially vanishing at infinity}, that is, for every $\eps>0$ there exists a $T_\eps>0$ such that for a.a.\ $t\in\R\setminus (-T_\eps,T_\eps)$ one has $\abs{x(t)}<\eps$ for all $x\in\sF$. 
	\end{enumerate}
\end{proposition}
\begin{proof}
	$(\Rightarrow)$ Let $\sF$ be relatively compact in $L_0^\infty(\R)$. Then $\sF$ is also relatively compact in $L^\infty(\R)$. Consequently, \lref{lemcomp} implies that $\sF$ is bounded and uniformly equimeasurable on $\R$. Let $\eps>0$. Now, observe that the relatively compactness of $\sF$ in $L_0^\infty(\R)$ implies the existence of a finite number of functions $x_1,\ldots,x_k\in \sF$ such that 
 \begin{equation*}
     \sF\subset \bigcup_{i=1}^k B_{\tfrac{\eps}{2}}(x_i),
 \end{equation*}
 where $B_{\eps/2}(x_i)\subset L_0^\infty(\R)$ is an open ball centered $x_i$ with radius $\tfrac{\eps}{2}$. What is more, there exists $T_\eps>0$ such that $|x_i(t)|<\tfrac{\eps}{2}$ for a.a.\ $t\in \R\setminus (-T_\eps,T_\eps)$ and $1\leq i\leq k$. Let $x\in\sF$. Then there exists $1\leq i_0\leq k$ such that $x\in B_{\eps/2}(x_{i_0})$ and thus
 \begin{equation*}
     |x(t)|\leq |x(t)-x_{i_0}(t)|+|x_{i_0}(t)|<\tfrac{\eps}{2}+\tfrac{\eps}{2}=\eps
 \end{equation*}
 for a.a.~$t\in \R\setminus (-T_\eps,T_\eps)$, which proves the condition (iii). 
 
	$(\Leftarrow)$ Assume the set $\sF$ satisfies the conditions (i$-$iii). Then \lref{lemcomp} ensures that $\sF$ is relatively compact in $L^\infty(\R)$. By sequential compactness, if $(x_n)_{n\in\N}$ is a sequence in $\sF$ this implies that there exist sequences $(k_n)_{n\in\N}$ and $x\in L^\infty(\R)$ such that $\lim_{n\to\infty}\norm{x_{k_n}-x}_\infty=0$. It remains to show $x\in L_0^\infty(\R)$. Thereto, given any $\eps>0$ there exists a $N\in\N$ such that $\norm{x_{k_n}-x}_\infty<\tfrac{\eps}{3}$ for all $n\geq N$, while (iii) implies that there is a $T>0$ with $\abs{x_{k_n}(t)}<\tfrac{\eps}{3}$ for a.a.\ $\abs{t}\geq T$ and for all $n\in\N$. This results in
	\begin{align*}
		\abs{x(t)}
		&\leq
		\abs{x(t)-x_{k_n}(t)}+\abs{x_{k_n}(t)}
		\leq
		\norm{x-x_{k_n}}_\infty+\abs{x_{k_n}(t)}
		<
		\eps
		\fall n\geq N
	\end{align*}
	and a.a.\ $t\in\R\setminus(-T,T)$, which means $\lim_{t\to\pm\infty}x(t)=0$, as desired
\end{proof}

As an immediate application, we obtain:
\begin{corollary}[compactness of multiplication operators]
    \label{cormult}
    If $A:\R\to\R^{d\tm d}$ is essentially bounded and satisfies the limit relations $\lim_{t\to\pm\infty}A(t)=0$, then the multiplication operator $M:W_0^{1,\infty}(\R)\to L_0^\infty(\R)$ pointwise defined as
    \begin{equation}
        [Mx](t):=A(t)x(t)\quad\text{for a.a.\ }t\in\R
        \label{cormult1}
    \end{equation}
    is compact. 
\end{corollary}
\begin{proof}
    Let $\sB$ denote the closed unit ball in $W_0^{1,\infty}(\R)$. We have to establish that the image $\sF:=M\sB\subset L_0^\infty(\R)$ is relatively compact. For this, we verify that all assumptions of \pref{propcom1} are satisfied and thereto abbreviate $C:=\esssup\limits_{t\in\R}|A(t)|$. 

    \underline{ad (i)} For every $y\in\sF$ there is a $x\in\sB$ such that
\begin{equation*}
\abs{y(t)}=\abs{A(t)x(t)}\leq C\norm{x}_{1,\infty}\leq C \text{ for a.a. }t\in\R.
\end{equation*}

    \underline{ad (ii)} Let $\eps>0$. Since $\lim\limits_{t\to\pm\infty}A(t)=0$, there exists $T_\eps>0$ such that 
    $$
        |A(t)|<\tfrac{\eps}{4}\quad\text{for a.a.\ }t\in \R\setminus (-T_\eps,T_\eps). 
    $$
    Note that a singleton $\{A|_{[-T_\eps,T_\eps]}\}$ is compact in $L^{\infty}([-T_\eps,T_\eps],\R^{d\tm d})$ and consequently \lref{lemcomp} implies the existence of a partition $\bigl\{\widetilde{P}_1,\ldots,\widetilde{P}_m\bigr\}$ of $[-T_\eps,T_\eps]$ such that for each $i\in\set{1,\ldots,m}$ and for a.a.\ $s,t\in \widetilde{P}_i$ one has $\abs{A(t)-A(s)}<\tfrac{\eps}{4}$. For $x\in\sB$ the Fundamental Theorem of Calculus \cite[p.~85, Thm.~3.30]{leoni:09} implies
    $$
        |x(t)-x(s)|
        \abs{\int_s^t\dot{x}(r)\d r}
        \leq
        \int_s^t\left |\dot{x}(r)\right|\d r
        \leq 
        \int_s^t\|\dot{x}\|_{\infty}\d r
        \leq
        |t-s|\fall t,s\in\R.
    $$
    Hence, the classical Arzel\`{a}-Ascoli theorem \cite[p.~95, 1.19c]{zeidler:95} yields that the family $B_\eps:=\set{u|_{[-T_\eps, T_\eps]}:\,u\in\sB}$ of functions from $\sB$ restricted to $[-T_\eps,T_\eps]$ is compact in $C([-T_\eps, T_\eps])$. Hence, there exists $\delta>0$ such that 
    $$
        |t-s|<\delta
        \Rightarrow
        |x(t)-x(s)|<\tfrac{\eps}{4C}\fall t,s\in [-T_\eps,T_\eps],\,x\in\sB.
    $$
    For positive integers $n>1+2T_\eps/\delta$ we now abbreviate
    \begin{align*}
        \bar{P}_i:&=\left[-T_\eps+\tfrac{2T_\eps}{n-1}(i-1),-T_\eps+\tfrac{2T_\eps}{n-1}i\right] \text{ for }1\leq i<n,&
        P_n:&=\R\setminus (-T_\eps,T_\eps)
    \end{align*}
    and define the partition $\left\{P_{ij}:=\widetilde{P}_i\cap \bar{P}_j\mid 1\leq i\leq m,\,1\leq j<n\right\}\cup \{P_n\}$, whose empty intersections are neglected. On the one hand, one has
    \begin{align}\label{infinity}
 \begin{split}
		\abs{y(t)-y(s)}&=|A(t)x(t)-A(s)x(s)|
		\leq
		\abs{A(t)} \abs{x(t)}+\abs{A(s)} \abs{x(s)}\\
		&\leq
		\abs{A(t)} \|x\|_{1,\infty}+\abs{A(s)} \|x\|_{1,\infty}
		\leq
		\tfrac{\eps}{4}+\tfrac{\eps}{4}<\eps
   \end{split}
	\end{align}
	for a.a.\ $t,s\in P_n$ and on the other hand it is
  \begin{align*}
		\abs{y(t)-y(s)}&=|A(t)x(t)-A(s)x(s)|
		\leq
		|A(t)-A(s)||x(t)|+|A(s)| |x(t)-x(s)|\\
		&\leq
		|A(t)-A(s)|\|x\|_{1,\infty}+C|x(t)-x(s)|\\
		&\leq
		|A(t)-A(s)|+C|x(t)-x(s)|
		\leq
		\frac{\eps}{4}+C\frac{\eps}{4C}<\eps
  \quad\text{for a.a.\ }t,s\in P_{ij}.
	\end{align*}

    \underline{ad (iii)} This follows directly from \eqref{infinity}. 

    In conclusion, \pref{propcom1} implies that $T\sB$ is relatively compact in $L^{\infty}_0(\R)$. 
\end{proof}

We continue with two further compactness criteria: 
\begin{corollary}[compactness in $W^{1,\infty}(I)$]
    \label{corW1coma}
    Let $I\subseteq\R$ be an interval. A subset $\sF\subseteq W^{1,\infty}(I)$ is relatively compact, if and only if the following holds: 
    \begin{itemize}
        \item[(i)] $\sF$ is bounded, 

        \item[(ii)] for every $\eps>0$ there exists a partition $\set{P_1,\ldots,P_n}$ of $I$ such that for each $j\in\set{1,\ldots,n}$ and for a.a.\ $s,t\in P_j$ one has $\abs{\dot x(t)-\dot x(s)}<\eps$ for all $x\in\sF$.
    \end{itemize}
\end{corollary}
\begin{proof}
    $(\Rightarrow)$ If $\sF\subseteq W^{1,\infty}(I)$ is relatively compact, then (i) holds. In order to establish (ii) we note that for every $\eps>0$ there exist functions $x_1,\ldots,x_k\in\sF$ such that $\sF\subseteq\bigcup_{i=1}^kB_{\eps/3}(x_i)$. Consequently, one the one hand, for every $x\in\sF$ there exists an $i\in\set{1,\ldots,k}$ such that $\norm{x-x_i}_{1,\infty}<\tfrac{\eps}{3}$. On the other hand, there is a partition $\set{P_1^1,\ldots,P_{n_i}^i}$ of the interval $I$ such that $\abs{\dot x_i(t)-\dot x_i(s)}<\tfrac{\eps}{3}$ for a.a.\ $t,s\in P_i$ and 
    \begin{align*}
        \abs{\dot x(t)-\dot x(s)}
        &\leq
        \abs{\dot x(t)-\dot x_i(t)}+
        \abs{\dot x_i(t)-\dot x_i(s)}+
        \abs{\dot x_i(s)-\dot x(s)}\\
        &\leq
        2\norm{x-x_i}_{1,\infty}+
        \abs{\dot x_i(t)-\dot x_i(s)}
        <
        \tfrac{2\eps}{3}+\tfrac{\eps}{3}=\eps. 
    \end{align*}
    Then the existence of a finite partition $\set{P_1,\ldots,P_n}$ as claimed in (ii) results by refining the partitions $\set{P_1^i,\ldots,P_{n_i}^i}$, $i\in\set{1,\ldots,k}$.

    $(\Leftarrow)$ Due to condition (i) we have that $\sF$ is equicontinuous and bounded in $L^\infty(I)$ and \pref{propcom1} yields that $\sF$ is compact in $L^\infty(I)$. But this implies the relative compactness of $\sF$ in $W^{1,\infty}(I)$. 
\end{proof}

\begin{corollary}[compactness in $W_0^{1,\infty}(\R)$]
    \label{corW1comb}
    A subset $\sF\subseteq W_0^{1,\infty}(\R)$ is relatively compact, if and only if the following holds: 
    \begin{itemize}
        \item[(i)] $\sF$ is bounded, 

        \item[(ii)] for every $\eps>0$ there exists a partition $\set{P_1,\ldots,P_n}$ of $\R$ such that for each $j\in\set{1,\ldots,n}$ and a.a.\ $s,t\in P_j$ one has $\abs{\dot x(t)-\dot x(s)}<\eps$ for all $x\in\sF$, 

        \item[(iii)] for every $\eps>0$ there exists a $T_\eps>0$ such that for a.a.\ $t\in\R\setminus(-T_\eps,T_\eps)$ one has $\max\set{\abs{x(t)},\abs{\dot x(t)}}<\eps$ for all $x\in\sF$.
    \end{itemize}
\end{corollary}
\begin{proof}
    $(\Rightarrow)$ If $\sF\subseteq W_0^{1,\infty}(\R)$ is relatively compact, then the statements (i) and (ii) result as in the proof of \cref{corW1coma}, while (iii) is established as in the proof of \pref{propcom1}. 

    $(\Leftarrow)$ Using \cref{corW1coma} we conclude that $\sF$ is relatively compact in $W^{1,\infty}(\R)$. Then, given a sequence $(x_n)_{n\in\N}$ in $\sF$,  there exists a convergent subsequence $(x_{k_n})_{n\in\N}$, i.e.\ there is a $x\in W^{1,\infty}(\R)$ with $\lim_{n\to\infty}\norm{x_{k_n}-x}_{1,\infty}=0$. Given $\eps>0$ there is a $N\in\N$ with $\norm{x_{k_n}-x}_{1,\infty}<\tfrac{\eps}{2}$ for all $n\geq N$ and a $T_\eps>0$ with $\max\set{\abs{x_{k_n}(t)},\abs{\dot x_{k_n}(t)}}<\tfrac{\eps}{2}$ for a.a.\ $\abs{t}\geq T_\eps$. In conclusion, 
    \begin{align*}
        \abs{x(t)}
        &\leq
        \abs{x(t)-x_{k_n}(t)}+\abs{x_{n_k}(t)}
        <
        \eps\quad\text{for a.a.\ }t\in\R\setminus(-T_\eps,T_\eps)
    \end{align*}
    and similarly for $\dot x$. Therefore, $x\in W_0^{1,\infty}(\R)$ is verified. 
\end{proof}

Having established this, our further approach crucially depends on notions from Topological Dynamics collected in App.~\ref{appA}. Thereto, rather than the Carath{\'e}odory equation \eqref{cde} we consider its equation of perturbed motion w.r.t.\ the branch $(\phi_\lambda)_{\lambda\in\Lambda}$, namely
\begin{equation}
    \dot x=\tilde f(t,x,\lambda)
    \label{cdep}
\end{equation}
having the right-hand side $\tilde f:\Omega_0\to\R^d$, 
$$
    \tilde f(t,x,\lambda):=f(t,x+\phi_\lambda(t),\lambda)-f(t,\phi_\lambda(t),\lambda)
$$
defined on $\Omega_0:=\set{(t,x,\lambda)\in\R\tm\R^d\tm\Lambda \mid\,x+\phi_\lambda(t)\in\Omega}$ and the trivial solution. 
%\comment{CP: Does this definition of the domain for $\tilde f$ make sense?
%\newline \textcolor{red}{RS: I replaced $\R\times\Omega_0$ by $\Omega_0$}
%\newline CP: added $\R\tm$
%\newline IPL: can we remove the "for all $t\in\R$"? It sounds somewhat cumbersome since the triple $(t,x,\lambda)\in\R\tm\R^d\tm\Lambda$ appears on the left. 
%\newline CP: I would say yes and corrected!
%}

In this setting, the \emph{hull} of a Carath{\'e}odory equation \eqref{cdep} depending on $\lambda\in\Lambda$ is denoted as $\sH(\lambda)$, while $\alpha(\lambda),\omega(\lambda)\subseteq\sH(\lambda)$ abbreviate the corresponding $\alpha$- resp.\ $\omega$-limit sets. With fixed $\lambda\in\Lambda$, a subset $\sG\subseteq\sH(\lambda)$ is said to be \emph{admissible}, if the following holds:
\begin{itemize}
	\item $Z_\sG:=\set{x\in\R^d\mid\exists g\in\sG:\,g(t,x)=0\text{ for a.a.\ }t\in\R}$ is compact and totally disconnected, 

	\item for each function $g\in\sG$ the set $\{\phi\in L^\infty(\R)\mid\phi\text{ is strongly differentiable with }$ $\dot\phi(t) =g(t,\phi(t))\text{ a.e.\ in }\R\}$ consists only of constant functions.
\end{itemize}

In the proof of \pref{proppropg} below, the inclusion $0\in Z_\sG$ will be important, i.e.\ whether the zero solution solves $\dot x =g(t,x)$ with $g\in\sH(\lambda)$. This issue is tackled in
\begin{lemma}
    \label{rmk:null-sol-hull}   
    For all $\lambda\in\Lambda$ and $g\in\sH(\lambda)$ one has $g(t,0)=0$ a.e.\ in $\R$.
\end{lemma}
\begin{proof}
    Let $\lambda\in\Lambda$ be fixed. Above all, the equation of perturbed motion \eqref{cdep} has the trivial solution. We establish that it also solves $\dot x=g(t,x)$ for $g\in\sH(\lambda)$. On the one hand, this is evident for any $S^s\tilde f(\cdot,\lambda)\in\sH(\lambda)$ with $s\in\R$, where $S^s\tilde f(t,x,\lambda):=\tilde f(t+s,x,\lambda)$. On the other hand, if $g\in\sH(\lambda)$, there is $(t_n)_{n\in\N}$ such that $f_n\xrightarrow{\sigma_\Q}g$ (cf.~App.~\ref{appA}), that is, for any interval $[q_1,q_2]\subset\R$, $q_1,q_2\in\Q$, 
    \[
        \int_{q_1}^{q_2} g(t,0)\d t
        =
        \lim_{n\to\infty}\int_{q_1}^{q_2}S^{t_n}\tilde f(t,0,\lambda)\d t
        =
        0. 
    \]
    The above formula can be extended to any bounded interval in $\R$ due to the additivity and absolute continuity of Lebesgue's integral and the density of $\Q$ in $\R$. As a consequence, the identically null function is also a solution for $\dot x =g(t,x)$ for all $g\in\sH(\lambda)$. 
\end{proof}

\begin{proposition}
	\label{proppropg}
    Let $\lambda\in\Lambda$. If $\alpha(\lambda)\cup\omega(\lambda)\subseteq\sH(\lambda)$ is admissible, then the mapping $G(\cdot,\lambda):W_0^{1,\infty}(\R)\to L_0^\infty(\R)$ is proper on all bounded, closed subsets of $W_0^{1,\infty}(\R)$. 
\end{proposition}
\begin{proof}
    The proof follows the arguments yielding \cite[Lemma~2.2]{poetzsche:skiba:20b}, but some steps need to be adjusted to the present more general setting of $W_0^{1,\infty}(\R)$ and $L_0^\infty(\R)$. 
    We neglect the dependence on the fixed $\lambda\in\Lambda$ in our notation and simply write $f$ for $\tilde f$. Due to \lref{lemcharprop} with spaces $X=W_0^{1,\infty}(\R)$ and $Y=L_0^\infty(\R)$, the claim, namely $G(\cdot,\lambda)|_{\sB}$ is proper on each bounded, closed $\sB\subset W_0^{1,\infty}(\R)$, is equivalent to the implication
	\begin{multline*}
		\text{$(G(x_n))_{n\in\N}$ converges in $L_0^\infty(\R)$ for some bounded sequence $(x_n)_{n\in\N}$ in $W_0^{1,\infty}(\R)$}\\
		\Rightarrow\text{$(x_n)_{n\in\N}$ has a convergent subsequence in $W_0^{1,\infty}(\R)$.}
	\end{multline*}
	Accordingly, let $(x_n)_{n\in\N}$ be a bounded sequence in $W_0^{1,\infty}(\R)$ such that $(G(x_n))_{n\in\N}$ converges in $L_0^\infty(\R)$ to some function $y$. 
 We thus need to show the existence of a convergent subsequence $(x_{n_k})_{k\in\N}$ in $W_0^{1,\infty}(\R)$. 
 We start by proving the existence of a convergent subsequence in $C_0(\R)\supset W_0^{1,\infty}(\R)$ by means of \lref{lem:comp_in_C_Z} with $\sF:=\set{x_n}_{n\in\N}$. Thereto, we aim to prove that the points (i--iii) in \lref{lem:comp_in_C_Z} are satisfied. The assumption of boundedness of $(x_n)_{n\in\N}$ in $W_0^{1,\infty}(\R)$ implies that there exists $R\geq 0$ with
	\begin{equation}
		\max\{\abs{x_n(t)},|\dot x_n(t)|\}< R\fall n\in\N, \text{ and a.a. }t\in\R.
		\label{lemproperO1}
	\end{equation}
	\underline{ad (i)}: From \eqref{lemproperO1} one obtains $\{x_n\mid n\in\N\}$ is a bounded subset in $C_0(\R)\subset L^\infty_0(\R)$.\\
	\underline{ad (ii)}: The mean value estimate implies
	\begin{equation}
		\abs{x_n(t)-x_n(s)}\stackrel{\eqref{lemproperO1}}{\leq}R\abs{t-s}\fall n\in\N,\,t,s\in\R
		\label{nouniequi}
	\end{equation}
	and therefore $\sF$ is uniformly equicontinuous.\\
	\underline{ad (iii)}: $Z:=\set{x\in\R^d\mid\exists g\in\alpha(\lambda)\cup\omega(\lambda):\,g(t,x)\equiv 0\text{ on }\R}$ is  compact and totally disconnected by the admissibility assumption, while $0\in Z$ as clarified in \lref{rmk:null-sol-hull}.

We choose a sequence in $\sF\subset C_{\{0\}}(\R)\subset C_{Z}(\R)$, which clearly is a subsequence of $(x_n)_{n\in\N}$ and w.l.o.g.\ denoted as $(x_n)_{n\in\N}$ again. 
For a real sequence $(s_n)_{n\in\N}$ with $\abs{s_n}\to\infty$, let us suppose that $\xi_n:=S^{s_n}x_n\in W_0^{1,\infty}(\R)$ converges pointwise to some function $\bar x\in BC(\R)$. 
We now show that $\bar x\in C_Z(\R)$ and $\bar x(\R)\subset Z$. Abbreviating the shifts $f_n:=S^{s_n}f$, 
%and $F_n\colon W_0^{1,\infty}(\R)\to L^\infty_0(\R)$, $F_n(x):=f_n(\cdot,x(\cdot))$,
	we obtain
	\begin{align*}
		\dot\xi_n(t)-f_n(t,\xi_n(t))
		&\equiv
		\dot x_n(t+s_n)-f(t+s_n, x_n(t+s_n))
		\equiv
		G(x_n)(t+s_n)\on\R
	\end{align*}
	and consequently
	\begin{equation}
		%\dot\xi_n=F_n(\xi_n)+S^{s_n}G(x_n)
        \dot\xi_n(t)=f_n(t,\xi_n(t))+S^{s_n}G(x_n)(t)
		\fall n\in\N \text{ and a.a. }t\in\R.
		\label{no36}
	\end{equation}
 
	(I) Claim: \emph{There exists some $f_0\in\LC$ such that for any bounded interval $I\subset\R$, the following convergence holds true,}
 \[
 \lim_{n\to\infty}\left|\int_I\big(f_n(t,\xi_n(t))-f_0(t,\bar x(t))\big)\d t \right|=0.
 \]
 \\
	First, $(\xi_n)_{n\in\N}$ in $W_0^{1,\infty}(\R)$ is bounded and like in \eqref{nouniequi} also uniformly equicontinuous. 
 Hence, on every compact subset $J\subset\R$ the Ascoli--Arzel{\'a} theorem (see \cite[p.~95, 1.19c]{zeidler:95}) applies and $(\xi_n)_{n\in\N}$ converges compactly, up to a subsequence, to $\bar x$. Second, due to \lref{lem:Top-dyn-Hsigma} there is $f_0\in\LC$ such that, up to a subsequence, $(f_n)_{n\in\N}$ converges to $f_0\in \alpha(\lambda)\cup\omega(\lambda)$ in $(\LC,\sigma_\Q)$, that is, the following holds true
  \[
 \lim_{n\to\infty}\left|\int_I\big(f_n(t,x)-f_0(t,x)\big)\d t \right|=0\quad \text{for all } x\in \Q^d,
 \]
for all $I=[q_1,q_2]$, with $q_1,q_2\in\Q$. On the other hand, \lref{lem:equiv-top} guarantees that for all $I=[q_1,q_2]$, with $q_1,q_2\in\Q$,
\begin{equation}\label{eq:21/02-16:54}
 \lim_{n\to\infty}\sup_{x\in \sK^I}\left|\int_I\big(f_n(t,x(t))-f_0(t,x(t))\big)\d t \right|=0,
\end{equation}
where $\sK=\{\xi_n|_I\mid n\in\N\}\cup\{\bar x|_I\}$ is compact in $C(I)$ as proved above. Note that in fact \eqref{eq:21/02-16:54} immediately extends to any bounded interval $I\subset \R$ thanks to the additivity and absolute continuity of Lebesgue's integral and the density of $\Q$ in $\R$.
%Choose $\rho>0$ so large that $\bar B_\rho(0)\subset\R^d$ contains the ranges of $\xi_n$ and $\bar x$. 
Now consider 
\[
\begin{split}
  & \left|\int_I\big(f_n(t,\xi_n(t))-f_0(t,\bar x(t))\big)\d t \right|\\
  &\qquad\qquad\le \left|\int_I\big(f_n(t,\xi_n(t))-f_n(t,\bar x(t))\big)\d t \right|+\left|\int_I\big(f_n(t,\bar x(t))-f_0(t,\bar x(t))\big)\d t \right|\\
  &\qquad\qquad\le \int_I\left|\big(f_n(t,\xi_n(t))-f_n(t,\bar x(t))\big)\right|\d t +\left|\int_I\big(f_n(t,\bar x(t))-f_0(t,\bar x(t))\big)\d t \right|\\
  &\qquad\qquad\le m_R\lambda_1(I)\norm{\xi_n-\bar x}_{L^\infty(I)}+\left|\int_I\big(f_n(t,\bar x(t))-f_0(t,\bar x(t))\big)\d t \right|,
\end{split}
\]
where $m_R$ is the common Lipschitz coefficient in $x$ on the ball $\bar B_R(0)$ for all the functions in $\{f_n\mid n\in\N\}\cup\{f_0\}\subset\LC$ (which particularly is a subset of the hull of $f$). Such a common constant $m_R$ exists thanks to~\lref{lem:Top-dyn-Hsigma}(a) given the assumption \eqref{h0b}.
Taking the limit as $n\to\infty$ on both sides of the previous chain of inequalities, we obtain the claimed assertion thanks to the uniform convergence on compact intervals of  $(\xi_n)_{n\in\N}$ to $\bar x$ and thanks to \eqref{eq:21/02-16:54}.

    (II) Claim: \emph{For every $I\subset\R$ bounded, $\norm{S^{s_n}G(x_n)}_{L^\infty(I)}\to0$ as $n\to\infty$.}\\
	This follows readily from the inequality
\begin{equation*}
\abs{S^{s_n}G(x_n)(t)}\leq\abs{G(x_n)(t+s_n)-y(t+s_n)}+\abs{y(t+s_n)}\quad\text{for a.a.~} t\in\R,
\end{equation*}
where $y$ is the limit of $(G(x_n)_{n\in\N}$ in $L^\infty_0(\R)$ as fixed at the beginning of the proof.\\
	\indent (III) Now consider $t\in\R$ and $h>0$. Recall that since $\xi_n\in W_0^{1,\infty}(\R)$ for all $n\in\N$, they are absolutely continuous on bounded intervals and satisfy the Fundamental Theorem of Calculus \cite[p.~85, Thm.~3.30]{leoni:09}. Then, using also \eqref{no36}, we can write,
 \[
 \begin{split}
     \xi_n(t+h)-\xi_n(t)=\int_t^{t+h}\dot \xi_n(s)\d s=\int_t^{t+h}f_n(s,\xi_n(s))\d s+\int_t^{t+h}S^{s_n}G(x_n)(s)\d s.
 \end{split}
 \]
Taking the limit as $n\to \infty$ on both sides of the previous formula, and using (I), (II), and the fact that $(\xi_n)_{n\in\N}$ converges compactly to $\bar x$, we obtain
\[
    \bar x(t+h)-\bar x(t)=\int_t^{t+h}f_0(s,\bar x(s))\d s.
\]
Thus, dividing both sides by $h$ and taking the limit as $h\to0$, we have that $\bar x$ is differentiable almost everywhere and moreover solves the Carath\'eodory equation $\dot x= f_0(t,x)$. In fact, the solution identity $\dot{\bar x}(t)\equiv f_0(t,\bar x(t))$ a.e.~on $\R$ even guarantees $\bar x\in W^{1,\infty}(\R)$ by $(H_0)$.
Thus, the assumed admissibility of $\alpha(\lambda)\cup\omega(\lambda)$ enforces the function $\bar x$ to be a constant $x_0\in\R^d$ and consequently
$$
    f_0(t,x_0)\equiv f_0(t,\bar x(t))\equiv\dot{\bar x}(t)\equiv 0 \text{ for a.a. } t\in\R.
$$
Hence, by definition it is $\bar x(\R)\subseteq Z$.\newline\indent 
In summary, we verified the conditions (i--iii) of \lref{lem:comp_in_C_Z} and thus the sequence $(x_{n})_{n\in\N}$ converges, up to a subsequence,  in $C_0(\R)$ to a function $\widetilde{x}$. 
It remains to show convergence in the $W^{1,\infty}(\R)$-topology. On the one hand,  since for a.a.~$t\in\R$ the function $f$ is Lipschitz continuous with respect of $x$ (cf.~\eqref{h0b}), then for a.a.~$t\in\R$ one has
\begin{align*}
|f(t, x_n(t))-f(t, \widetilde{x}(t))|\le m_R\abs{x_n(t)-\widetilde x(t)}\le m_R\norm{x_n-\widetilde x}_\infty.
\end{align*}
 On the other hand, since 
  \[
 \begin{split}
     x_n(t+h)-x_n(t)=\int_t^{t+h}\dot x_n(s)\d s=\int_t^{t+h}f(s,x_n(s))\d s+\int_t^{t+h}G(x_n)(s)\d s
 \end{split}
 \]
 and since $(G(x_n))_{n\in\N}$ is assumed to converge uniformly to $y$ in $L^\infty_0(\R)$, reasoning as before, we conclude that $\widetilde x$ is differentiable almost everywhere and the derivative $\dot {\widetilde x}$ satisfies $\dot{\widetilde x}(t)=f(t,\widetilde{x}(t))+y(t)$ for a.a.~$t\in\R$. Finally, due to the above considerations, we have that for a.a.~$t\in\R$
 \[
 \begin{split}
      |\dot x_n(t)-\dot {\widetilde x}(t)|&\le |f(t,x_n(t))-f(t,\widetilde{x}(t))|+|G(x_n)(t)-y(t)|\\
      &\le m_R\norm{x_n-\widetilde x}_\infty+\norm{G(x_n)-y}_\infty,
 \end{split}
 \]
which implies that $(\dot x_n)_{n\in\N}$ converges to $\dot {\widetilde x}$ in $L^\infty_0(\R)$, and thus, $(x_n)_{n\in\N}$ converges to $\widetilde{x}$ in $W^{1,\infty}_0(\R)$. 
	This completes the proof.
\end{proof}

\begin{proposition}[properness]
	\label{propprop}
    If $\alpha(\lambda)\cup\omega(\lambda)\subseteq\sH(\lambda)$ is admissible for all $\lambda\in\Lambda$, then the operator $G:U\to L_0^\infty(\R)$ is proper on every product $\sB\tm\Lambda_0\subset U$ with bounded, closed $\sB\subset W_0^{1,\infty}(\R)$ and  $\Lambda_0\subseteq\Lambda$. 
 %  \comment{RS: I added the symbol $\subset U$
  %  \newline CP: Good!}
\end{proposition}
\begin{proof}
    In order to apply \lref{lemproper} with $X=W_0^{1,\infty}(\R)$ and $Y:=L_0^\infty(\R)$, we note that \tref{thmgprop} yields the continuity of $G$. First, given a bounded subset $\sB\subset W_0^{1,\infty}(\R)$, $\lambda_0\in\Lambda$ and $\eps>0$, we note that $(H_1)$ yields that there exists a $\delta>0$ such that
    $$
        \norm{G(x,\lambda)-G(x,\lambda_0)}_\infty
        \stackrel{\eqref{gdef}}{=}
        \esssup_{t\in\R}\abs{f(t,x(t),\lambda)-f(t,x(t),\lambda_0)}
        <
        \eps\fall x\in\sB
    $$
    and $\lambda\in B_\delta(\lambda_0)\cap\Lambda$. Thus, the set $\set{G(x,\cdot):\Lambda\to L_0^\infty(\R)\mid x\in\sB}$ is equicontinuous. Second, \pref{proppropg} ensures that each $G(\cdot,\lambda):W_0^{1,\infty}(\R)\to L_0^\infty(\R)$, $\lambda\in\Lambda$, is proper on closed, bounded subsets of $W_0^{1,\infty}(\R)$. Then \lref{lemproper} concludes the proof, because the compact subsets of $\R$ are just the closed and bounded ones. 
\end{proof}
\section{Global Evans functions}
\label{sec4}
The \emph{variational equations} corresponding to the solution branch $(\phi_\lambda)_{\lambda\in\Lambda}$ required in Hypothesis~$(H_1)$ read as
\begin{equation}
	\tag{$V_\lambda$}
	\dot x=D_2f(t,\phi_\lambda(t),\lambda)x.
	\label{var}
\end{equation}
Keeping a parameter $\lambda\in\Lambda$ fixed, note that these linear problems are again well-posed thanks to Hypothesis~$(H_1)$, meaning that in particular the map $t\mapsto D_2f(t,\phi_\lambda(t),\lambda)$ is locally integrable. Therefore, for each initial time $\tau\in\R$ there exists a unique solution $\Phi_\lambda(\cdot,\tau):\R\to\R^{d\tm d}$ to the initial value problem $\dot X=D_2f(t,\phi_\lambda(t),\lambda)X$, $X(\tau)=I_d$ in $\R^{d\tm d}$. We denote  $\Phi_\lambda(t,s)\in GL(\R^d)$ for $t,s\in\R$ as \emph{transition matrix} of \eqref{var}.

A solution $\phi_{\lambda}$ to \eqref{cde} is understood as \emph{hyperbolic} on an unbounded interval $I\subseteq\R$, if the associated variational equation $(V_{\lambda})$ has an \emph{exponential dichotomy} on $I$. This means there exist reals $K\geq 1$, $\alpha>0$ and a projection-valued function $P_{\lambda}:I\to\R^{d\tm d}$ such that $\Phi_{\lambda}(t,s)P_{\lambda}(s)=P_{\lambda}(t)\Phi_{\lambda}(t,s)$ and
\begin{align*}
	\abs{\Phi_{\lambda}(t,s)P_{\lambda}(s)}\leq Ke^{-\alpha(t-s)},\,
	\abs{\Phi_{\lambda}(s,t)[I_d-P_{\lambda}(t)]}\leq Ke^{-\alpha(t-s)}\text{ for all }s\leq t
\end{align*}
with $s,t\in I$ hold. Then $P_\lambda:I\to\R^{d\tm d}$ is called \emph{invariant projector}.

\begin{hypothesis*}[$\mathbf{H_2}$]
	For each parameter $\lambda\in\Lambda$ suppose that the bounded entire solution $\phi_\lambda$ to \eqref{cde} is hyperbolic on both $\R_+$ with projector $P_\lambda^+:\R_+\to\R^{d\tm d}$ and on $\R_-$ with projector $P_\lambda^-:\R_-\to\R^{d\tm d}$. We moreover assume there exists a $\lambda_0\in\Lambda$ such that
	\begin{equation}
		\dim\bigl(R(P_{\lambda_0}^+(0))\cap N(P_{\lambda_0}^-(0))\bigr)
		=
		\codim\bigl(R(P_{\lambda_0}^+(0))+N(P_{\lambda_0}^-(0))\bigr).
		\label{nomorse}
	\end{equation}
\end{hypothesis*}
\begin{proposition}
    \label{propfred}
	If $(H_0$--$H_2)$ hold, then $D_1G(x,\lambda)\in L(W_0^{1,\infty}(\R),L_0^\infty(\R))$ is Fredholm of index $0$ for all $(x,\lambda)\in U$ (with $U$ defined in \tref{thmgprop}(a)). 
\end{proposition}
\begin{proof}
	Let $\lambda\in\Lambda$. Because of $(H_2)$ the variational equation \eqref{var} has exponential dichotomies on both halflines, such that \cite[Thm.~2.6(b)]{poetzsche:skiba:24} implies the partial derivative $D_1G(0,\lambda)\in L(W_0^{1,\infty}(\R),L_0^\infty(\R))$ to be Fredholm. In particular, abbreviating $X_+:=R(P_{\lambda_0}^+(0))$ and $X_-:=N(P_{\lambda_0}^-(0))$ the index of $D_1G(0,\lambda_0)$ is given by
	$$
		\dim(X_+\cap X_-)-\dim(X_++X_-)^\perp
		=
		\dim(X_+\cap X_-)-\codim(X_++X_-)
		\stackrel{\eqref{nomorse}}{=}
		0
	$$
	due to \cite[step (V) in the proof of Thm.~2.6]{poetzsche:skiba:24}. Because the path $\mu\mapsto D_1G(0,\mu)$ is continuous on the connected parameter space $\Lambda$ we conclude that every derivative $D_1G(0,\lambda)$ has Fredholm index $0$. 
    Now let $x\in W_0^{1,\infty}(\R)\subset L_0^\infty(\R)$ and $\eps>0$. On the one hand, due to $(H_0)$ there is a $\delta>0$ with
	$$
		\abs{\xi}<\delta
		\quad\Rightarrow\quad
		\abs{D_2f(t,\xi+\phi_\lambda(t),\lambda)-D_2f(t,\phi_\lambda(t),\lambda)}
		<
		\eps
		\quad\text{for a.a.\ }t\in\R.
	$$
    %{\color{red} and all $\xi\in\R^d$
     On the other hand, there is a $T>0$ such that $x(t)\in B_\delta(0)$ holds for a.a.\ $t\in\R\setminus(-T,T)$ and thus $A(t):=D_2f(t,x(t)+\phi_\lambda(t),\lambda)-D_2f(t,\phi_\lambda(t),\lambda)$ satisfies
	$$
		\abs{A(t)}
		\leq
		\abs{D_2f(t,x(t)+\phi_\lambda(t),\lambda)-D_2f(t,\phi_\lambda(t),\lambda)}
		<
		\eps.%\quad{\color{red}\text{for all } t\in\R\setminus(-T,T)}.
	$$
    In conclusion, $A:\R\to\R^{d\tm d}$ is essentially bounded and satisfies $\lim_{t\to\pm\infty}A(t)=0$. Therefore, \cref{cormult} implies that the multiplication operator $M\in L(W_0^{1,\infty}(\R),L_0^\infty(\R))$ defined in \eqref{cormult1} is compact. We obtain that
	$$
		D_1G(x,\lambda)
        =
        D_1G(0,\lambda)+D_1G(x,\lambda)-D_1G(0,\lambda)
        =
        D_1G(0,\lambda)+M
	$$
 is a compact perturbation of a Fredholm operator $D_1G(0,\lambda)$ (with index $0$). Whence, \cite[pp.~295--296, Thm.~5.C]{zeidler:95} implies $D_1G(x,\lambda)\in F_0(W_0^{1,\infty}(\R),L_0^\infty(\R))$. 
\end{proof}

\begin{lemma}\label{lemcont}
    If $(H_0$--$H_2)$ hold, then the invariant projectors $P_\lambda^+$, $P_\lambda^-$ from Hypothesis $(H_2)$ can be chosen such that $\lambda\mapsto P_\lambda^+(0)$ and $\lambda\mapsto P_\lambda^-(0)$ are continuous on $\Lambda$. 
    %there exists a continuous family $\widetilde{P}_{\pm}\colon %\Lambda\to \mathcal{L}(\R^d)$ satisfying ED on %$\mathbb{R}_{\pm}$, where $\Lambda$ is $\sigma$-compact.
\end{lemma}
\begin{proof}
     For each $\lambda\in\Lambda$ we construct a function of projectors being continuous locally near $\lambda$ and then build them globally by means of a partition of unity. Indeed, referring to \cite[Lemma~3.1]{poetzsche:skiba:24} there exists a $r(\lambda)>0$ such that the invariant projectors $P_\mu^+$ for the exponential dichotomy on $\R_+$ can be chosen to be continuous in $\mu\in B_{r(\lambda)}(\lambda)$. During the course of this proof we denote $P\in\R^{d\tm d}$ as \emph{invariant projection} for \eqref{var}, provided that $\Phi_\lambda(\cdot,0)P\Phi_\lambda(0,\cdot):\R_+\to\R^{d\tm d}$ defines an invariant projector for the assumed exponential dichotomy of $(V_\lambda)$ on $\R_+$. 

    (I) We show that the convexity of the following set of projections
    $$
        \Pi(\lambda)
        :=
        \set{P\in\R^{d\tm d}\mid
        P\text{ is an invariant projection for }\eqref{var}}
        \fall\lambda\in\Lambda.
    $$
    Thereto, let $P,\bar P\in\Pi(\lambda)$ yielding dichotomies with respective growth rates $\alpha,\bar\alpha>0$ and dichotomy constants $K$, $\bar K\geq 1$. We first show 
$\theta P+\bar\theta\bar P\in \Pi(\lambda)$ for all $\theta,\bar\theta\in [0,1]$ with $\theta+\bar\theta=1$. Since $R(P)=R(\bar P)$ (the range of projectors for dichotomies on $\R_+$ is uniquely determined), it follows that $\theta R(P)+\bar\theta R(\bar P)=R(P)=R(\bar P)$. Thus, 
\begin{align*}
    P(\theta P+\bar\theta\bar P)&=\theta P+\bar\theta\bar P,&
    \bar P(\theta P+\bar\theta \bar P)&=\theta P+\bar\theta\bar P
\end{align*}
    and consequently
\begin{multline*}
[\theta P+\bar\theta\bar P][\theta P+\bar\theta\bar P]
=
\theta P(\theta P+\bar\theta\bar P)+\bar\theta \bar P(b P+\bar\theta \bar P)\\
=
\theta(\theta P+\bar\theta\bar P)+\bar\theta(\theta P+\bar\theta\bar P)
=
(\theta+\bar\theta)(\theta P+\bar\theta\bar P)=\theta P+\bar\theta\bar P
\end{multline*}
which implies that $\theta P+\bar\theta\bar P$ is a projection. It remains to derive the dichotomy estimates
\begin{align*}
&|\Phi_\lambda(t,0)[\theta P+\bar\theta\bar P]\Phi_\lambda(0,s)|
\leq
\theta|\Phi_\lambda(t,0)P\Phi_\lambda(0,s)|+\bar\theta|\Phi_\lambda(t,0)\bar P\Phi_\lambda(0,s)|\\
&\leq
\theta K\text{e}^{-\alpha(t-s)}+\bar\theta\bar K\text{e}^{-\bar\alpha(t-s)}
\leq
K_+\text{e}^{-\alpha_+(t-s)}\fall 0\leq s\leq t
\end{align*}
and
\begin{align*}
	&
	\abs{\Phi_\lambda(t,0)[I_d-\theta P-\bar\theta\bar P]\Phi_\lambda(0,s)}\
	=
	\abs{\Phi_\lambda(t,0)[(\theta+\bar\theta)I_d-\theta P+\bar\theta\bar P]\Phi_\lambda(0,s)}\\
	&=
	\abs{\theta \Phi_\lambda(t,0)[I_d-P]\Phi_\lambda(0,s)+\bar\theta \Phi_\lambda(t,0)[I_d-\bar P]\Phi_\lambda(0,s)}\\
&\leq
\theta|\Phi_\lambda(t,0)[I_d-P]\Phi_\lambda(0,s)|+\bar\theta|\Phi_\lambda(t,0)[I_d-\bar P]\Phi_\lambda(0,s)|\\
&\leq
\theta K\text{e}^{-\alpha(s-t)}+\bar\theta\bar K\text{e}^{-\bar\alpha(s-t)}
\leq
\theta K_+\text{e}^{-\alpha_+(s-t)}+\bar\theta K_+\text{e}^{-\alpha_+(s-t)}\\
&\leq
K_+\text{e}^{-\alpha_+(s-t)} \fall 0\leq t\leq s,
\end{align*}
where $K_+:=\max\{K,\bar K\}$ and $\alpha_+:=\min\{\alpha,\bar\alpha\}$. Finally, notice that the convexity of $\Pi(\lambda)$ implies the following additional property:
\begin{equation*}
P_1,\ldots,P_n\in\Pi(\lambda)\Longrightarrow \sum_{i=1}^n\theta_i P_i\in\Pi(\lambda)
\text{ with }\sum_{i=1}^n\theta_i=1, \theta_i\in [0,1].
\end{equation*}

    (II) The open interval $\Lambda$ can be covered by the family $\set{B_{r(\lambda)}(\lambda)\mid \lambda\in\Lambda}$ of open balls from Step (I). For each $\lambda\in\Lambda$ there is an invariant projector $P_\lambda^+:\R_+\to\R^{d\tm d}$ due to $(H_2)$ and we define the projection $P_\lambda:=P_\lambda^+(0)$. With \cite[Lemma~3.1]{poetzsche:skiba:24} it can be continued to a continuous function $P_\lambda:B_{r(\lambda)}(\lambda)\to\R^{d\tm d}$ such that each $P_\lambda(\mu)$ is an invariant projection for $(V_\mu)$. Since $\Lambda$ is $\sigma$-compact, there exists a countable cover
    \begin{equation*}
        \Lambda\subset \bigcup_{i\in\N}B_{r(\lambda_i)}(\lambda_i).
    \end{equation*}
    of $\Lambda$. W.l.o.g.\ we can assume that this cover is locally finite (otherwise, one can construct a locally finite cover $\{V_i\}_{i\in\N}$ of $\Lambda$ with  
$V_i\subset B_{r(\lambda_i)}(\lambda_i)$ for $i\in\N$, cf.~\cite[Thm.~5.2.1]{eng:89}). Let $\{\tau_i\colon\Lambda\to [0,1]\mid i\in \mathbb{N}\}$ be a partition of unity subordinated to the above cover, i.e., 
\begin{align*}
	\sum_{i=1}^{\infty}\tau_i(\lambda)&=1\fall\lambda\in\Lambda,&
	\overline{\{\lambda\in \Lambda\mid\tau_i(\lambda)\neq 0\}}\subset B_{r(\lambda_i)}(\lambda_i)
	\fall i\in\N.
\end{align*}
From the above properties it follows that if $\lambda\not\in B_{r(\lambda_i)}(\lambda_i)$, then $\tau_j(\lambda)=0$; we define
$$
    \tilde P(\lambda):=\sum_{i=1}^{\infty} \tau_i(\lambda)P_{\lambda_i}(\lambda)\fall\lambda\in\Lambda.
$$
Now on the one hand, step (II) guarantees that each $\tilde P(\lambda)$ defines an invariant projection for \eqref{var}. On the other hand, standard arguments based on corresponding properties of the partition of unity show that $\tilde P:\Lambda\to\R^{d\tm d}$ is continuous. 

    (III) Based on the fact that the kernels of the projectors are uniquely determined for dichotomies on $\R_-$, the argument for $P_\lambda^-$ is analogous. This completes the proof.
\end{proof}

Following \cite{poetzsche:skiba:24} the above preparations allow us to introduce a tool to detect changes in the set of bounded entire solutions to parametrized Carath{\'e}odory equations \eqref{cde}. Thereto, note that due to Hypothesis~$(H_2)$ the bounded entire solutions $\phi_\lambda$ are hyperbolic on both $\R_+$ with a projector $P_\lambda^+$ and on $\R_-$ with a projector $P_\lambda^-$. 
%Now thanks to \cite[Lemma~3.1]{poetzsche:skiba:24} there exists a $\rho_0>0$ such that the mappings $(t,\lambda)\mapsto P_\lambda^+(t)$ and $(t,\lambda)\mapsto P_\lambda^-(t)$ are continuous with
%\begin{align*}
%	r&\equiv\dim R(P_\lambda^+(0)),&
%	n&\equiv\dim N(P_\lambda^-(0))\on B_{\rho_0}(\lambda^\ast).
%\end{align*}
Consequently, there exist functions $\xi^+_1,\ldots,\xi^+_r$, $\xi^-_1,\ldots,\xi^-_n$ having the following properties for all $\lambda\in\Lambda$: 
\begin{itemize}
	\item $\xi^+_1(\lambda),\ldots,\xi^+_r(\lambda)$ is a basis of $R(P_\lambda^+(0))\subseteq\R^d$,

	\item $\xi^-_1(\lambda),\ldots,\xi^-_n(\lambda)$ is a basis of $N(P_\lambda^-(0))\subseteq\R^d$.
\end{itemize}
\begin{lemma}\label{lem50}
	If $(H_0$--$H_2)$ hold, then $r+n=d$. Moreover, the vectors $\xi^+_1(\lambda),\ldots,\xi^+_r(\lambda)$, $\xi^-_1(\lambda),\ldots,\xi^-_n(\lambda)$ form a basis of $\R^d$ if and only if $R(P_\lambda^+(0))\oplus N(P_\lambda^-(0))=\R^d$ holds for all $\lambda\in\Lambda$.
\end{lemma}
\begin{proof}
	Again, \cite[step (V) in the proof of Thm.~2.6]{poetzsche:skiba:24} guarantees $0=n-(d-r)$, i.e.\ $n+r=d$, because due to \pref{propfred} the index of the Fredholm operator $D_1G(0,\lambda)$ is $0$ for all $\lambda\in\Lambda$. The claimed equivalence is an easy consequence of Linear Algebra. 
\end{proof}
On this basis, a \emph{global Evans function} for \eqref{var} is defined by
\begin{align*}
	E:\Lambda&\to\R,&
E(\lambda)&:=\det\bigl(\xi^+_1(\lambda),\ldots,\xi^+_r(\lambda),\xi^-_1(\lambda),\ldots,\xi^-_n(\lambda)\bigr).
\end{align*}
Although they depend on the choice of the vectors $\xi^+_i(\lambda)$, $\xi^-_j(\lambda)\in\R^d$, any two Evans functions differ only by a product with a nonvanishing function (this factor is a determinant of the transformation matrices that describe the change of bases). 

\begin{proposition}[properties of global Evans functions]\label{propEvans}
	If $(H_0$--$H_2)$ hold, then there exists a continuous global Evans function $E:\Lambda\to\R$ of \eqref{var} and the following statements are equivalent for all $\lambda\in\Lambda$: 
	\begin{enumerate}[$(a)$]
		\item $E(\lambda)\neq 0$,

		\item $D_1G(0,\lambda)\in GL(W_0^{1,\infty}(\R),L_0^\infty(\R))$. 
    \end{enumerate}
\end{proposition}
\begin{proof}
	Thanks to \lref{lemcont} one can choose continuous mappings $\lambda\mapsto P_\lambda^+(0)$ and $\lambda\mapsto P_\lambda^-(0)$. Given this, the argument from \cite[Prop.~3.3]{poetzsche:skiba:24} yields that $\xi_i^+$, $\xi_j^-:\Lambda\to\R^d$ are continuous for $1\leq i\leq r$, $1\leq j\leq n$ and in turn the continuity of $E:\Lambda\to\R$. The claimed equivalence is due to \cite[Thm.~2.4 and Prop.~3.4]{poetzsche:skiba:24}. 
\end{proof}

A parameter $\lambda\in\Lambda$ is denoted as \emph{critical}, if the corresponding solution $\phi_\lambda$ is not hyperbolic on $\R$ and we introduce the \emph{set of critical values}
$$
	\fC:=\set{\lambda\in\Lambda:\,\phi_\lambda\text{ is not hyperbolic on }\R}
$$
for \eqref{var}. 
Note that $\fC\subseteq\R$ is closed in $\Lambda$, but not necessarily discrete. 
\begin{proposition}
	\label{propevans}
	Let $(H_0$--$H_2)$ hold. The set $\fC$ is closed in $\Lambda$ and characterized as
	\begin{align*}
		\fC
		=
		\set{\lambda\in\Lambda:\,R(P_\lambda^+(0))\cap N(P_\lambda^-(0))\neq\set{0}}
		=
		E^{-1}(\set{0}).
	\end{align*}
\end{proposition}
\begin{proof}
	Above all, the continuous path of operators
	\begin{align}
		T:\Lambda&\to F_0(W_0^{1,\infty}(\R),L_0^\infty(\R)),&
		T(\lambda)&:=D_1G(0,\lambda)
		\label{tpath}
	\end{align}
	is fundamental for our further analysis; it is well-defined due to \pref{propfred}. Now let $\lambda\in\Lambda$ be given. Since each $T(\lambda)$ is Fredholm of index $0$, the set $\fC$ allows the characterization $\fC=\set{\lambda\in\Lambda:\,N(T(\lambda))\neq\set{0}}$. As in \cite[Step (I) in the proof of Thm.~2.6]{poetzsche:skiba:24} one sees that $N(T(\lambda))$ and $R(P_\lambda^+(0))\cap N(P_\lambda^-(0))$ are isomorphic, yielding the first characterization of $\fC$. Moreover, 
	$
		\fC
		=
		\set{\lambda\in\Lambda:\,N(T(\lambda))\neq\set{0}}
		=
		\set{\lambda\in\Lambda:\,E(\lambda)=0}
	$
	results with Props.~\ref{propfred} and \ref{propEvans}. Finally, because $E:\Lambda\to\R$ is continuous due to \pref{propEvans}, we identify $\fC$ as closed in $\Lambda$ as preimage of the closed set $\set{0}$. 
\end{proof}
\section{Global bifurcations of homoclinic solutions}
\label{sec5}
Under our standing Hypotheses~$(H_0$--$H_2)$ the parametrized Carath{\'e}odory equations \eqref{cde} possess a continuous branch $(\phi_\lambda)_{\lambda\in\Lambda}$ of bounded entire solutions. We denote the graph of $\lambda\mapsto\phi_\lambda$, namely the set (see \fref{fig2a})
$$
	\cT:=\set{(\phi_\lambda,\lambda)\in W^{1,\infty}(\R,\Omega)\tm\R:\,\lambda\in\Lambda}
$$
as \emph{prescribed branch} for \eqref{cde}. Note that the inclusion $\phi_\lambda\in W^{1,\infty}(\R,\Omega)$ holds because of \cite[Thm.~2.3]{poetzsche:skiba:24}. Here, and from now on, we indicate subsets of $W^{1,\infty}(\R)\tm\R$ by calligraphic letters. 

An entire solution $\phi^\ast:=\phi_{\lambda^\ast}$ of $(C_{\lambda^\ast})$ is said to \emph{bifurcate} at $\lambda^\ast\in\Lambda$ from $(\phi_\lambda)_{\lambda\in\Lambda}$ (or $\cT$), if there exists a parameter sequence $(\lambda_n)_{n\in\N}$ in $\Lambda$ converging to $\lambda^\ast$ such that each Carath{\'e}odory equation $(C_{\lambda_n})$ possesses a bounded entire solution $\psi_n\neq\phi_{\lambda_n}$ with
$$
	\lim_{n\to\infty}\sup_{t\in\R}\abs{\psi_n(t)-\phi_{\lambda^\ast}(t)}=0; 
$$
%\comment{\textcolor{blue}{I suggest replacing $W^{1,\infty}(\R)$ in Figure 1 by the letter $X$ since it is applied in our article to both cases: $W^{1,\infty}(\R)$ and 
%$W^{1,\infty}_0(\R)$, see also the next page. }}
here $(\phi^\ast,\lambda^\ast)$ is called \emph{bifurcation point} for \eqref{cde}, while $\lambda^\ast$ is a \emph{bifurcation value}. If
$$
	\fB:=\set{\lambda\in\Lambda:\,(\phi_\lambda,\lambda)\text{ is a bifurcation point for }\eqref{cde}},
$$
then thanks to \cite[Thm.~4.1]{poetzsche:skiba:24} the inclusion $\fB\subseteq\fC$ holds. We define (see \fref{fig2a})
\begin{multline*}
	\cS
	:=
	\set{(\phi,\lambda)\in W^{1,\infty}(\R)\tm\Lambda:\,\phi\text{ solves }\eqref{cde},\phi\neq\phi_\lambda\text{ and }\phi-\phi_\lambda\in W_0^{1,\infty}(\R)}\\
	\cup
	\set{(\phi_\lambda,\lambda)\in W^{1,\infty}(\R)\tm\R:\,\lambda\in\fC}
	\subseteq
	W^{1,\infty}(\R,\Omega)\tm\Lambda
\end{multline*}
as set of all bounded entire solutions to \eqref{cde} being homoclinic to $\phi_\lambda$, or being critical. Note that a pair $(\phi_{\lambda^\ast},\lambda^\ast)$ is a bifurcation point for \eqref{cde}, if every neighborhood of $(\phi_{\lambda^\ast},\lambda^\ast)$ contains an element of $\cS$. 
\begin{SCfigure}[2]
	\includegraphics[scale=0.5]{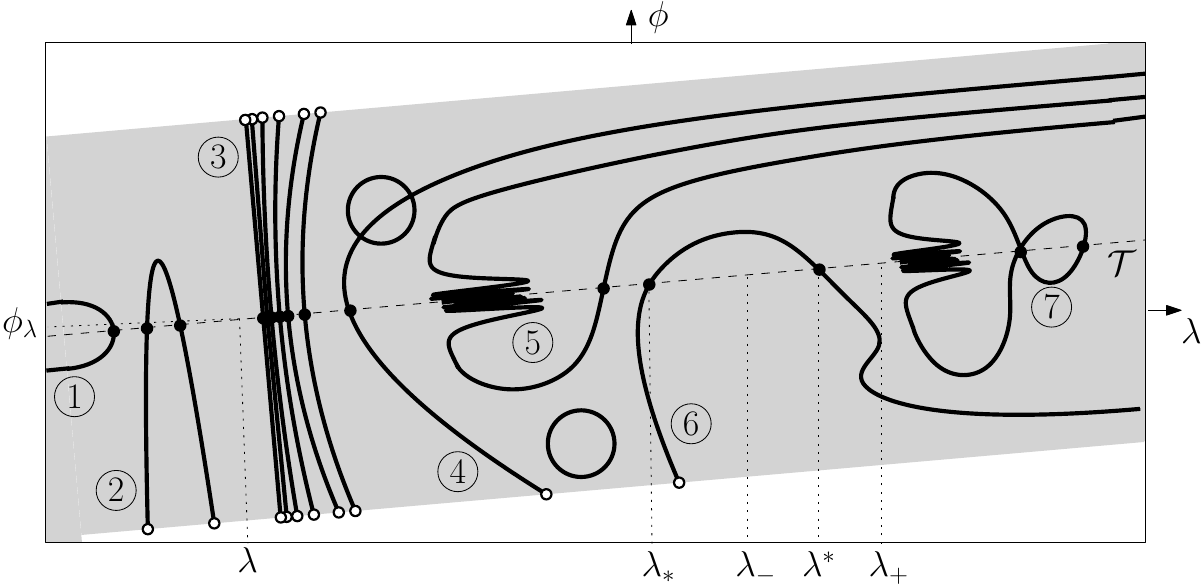}
	\caption{The sets $\cT$ of prescribed solutions $\phi_\lambda$ (dashed) and $\cS$ of all the non-prescribed solutions $\phi$ homoclinic to $\phi_\lambda$ (solid) for \eqref{cde} with $\Lambda=\R$, as well as an interval $(\lambda_-,\lambda_+)\subseteq\Lambda$ containing a bifurcation value $\lambda^\ast$.}
	\label{fig2a}
\end{SCfigure}

Eventually, having a global Evans function $E:\Lambda\to\R$ at hand, a sufficient condition for bifurcation, whose statement extends \cite[Thm.~4.2]{poetzsche:skiba:24}, is
\begin{theorem}[global bifurcation of solutions homoclinic to $\cT$]
	\label{thmbif}
	Let $(H_0$--$H_2)$ hold. If parameters $\lambda_-,\lambda_+\in\Lambda$ with $\lambda_-<\lambda_+$ satisfy $E(\lambda_-)E(\lambda_+)<0$, then there exists a bifurcation value $\lambda^\ast\in(\lambda_-,\lambda_+)$. More precisely, there exists a component $\cC$ of $\cS$ intersecting $\set{(\phi_\lambda,\lambda):\,\lambda\in(\lambda_-,\lambda_+)}\subseteq\cT$ in $(\phi_{\lambda^\ast},\lambda^\ast)$ such that at least one of the following alternatives (cf.~Fig.~\ref{fig2a}) holds:
	\begin{enumerate}
		\item[(a)] $\cC':=
        \{(\psi,\lambda)\in W_0^{1,\infty}(\R)\tm\Lambda:\,(\phi_\lambda+\psi,\lambda)\in\cC\}$ is noncompact in the set $U\subseteq W_0^{1,\infty}(\R)\tm\Lambda$ from \tref{thmgprop}(a), i.e.\ its closure $\overline{\cC'}$ in $W_0^{1,\infty}(\R)\tm\R$ does
		\begin{itemize}
			\item[$(a_1)$] contain a point $(x,\lambda)\in\partial U$, or 

			\item[$(a_2)$] is not a compact subset of $W_0^{1,\infty}(\R)\tm\R$, which in turn means that for each compact $K_0\subseteq W_0^{1,\infty}(\R)$ and for each compact $\Lambda_0\subseteq\R$ there exists a pair $(x,\lambda)\in\cC'\setminus(K_0\tm\Lambda_0)$, 
		\end{itemize}

		\item[(b)] $\cC$ contains a point $(\phi_{\lambda_\ast},\lambda_\ast)$ with $\lambda_\ast\in\fB\setminus[\lambda_-,\lambda_+]$. 
	\end{enumerate}
	In particular, for $\Omega=\R^d$ and $\Lambda=\R$ the alternative $(a)$ reduces to $(a_2)$. 
\end{theorem}
As explained in \cite{poetzsche:skiba:24}, the assumptions of \tref{thmbif} (and of the subsequent results in this section) can only be fulfilled in $d>1$ dimensions and for nontrivial projections, i.e.\ $P_\lambda^+(t),P_\lambda^-(t)\not\in\set{0_d,I_d}$. 

We emphasize that the alternatives in \tref{thmbif} are nonexclusive and the circled numbers in \fref{fig2a} illustrate possible shapes of the component $\cC'$: 
%comment{
%\textcolor{red}{RS: I made the corrections. }\newline
%\textcolor{blue}{RS: Does Set 3 from Figure 1 contain finitely many %components or infinitely many? If infinitely, then $(a_2)$ in 3: is OK; %otherwise, if, for instance, $X$ is finite-dimensional space, then Set 3 is %relatively compact and $(a_2)$ probably must be deleted.\newline RS: I %think Set 3 in Figure 1 satisfies $(a_1)$. Am I right?} 
%\newline {\color{red} IPL: I think set 3 is meant to represent infinitely %many otherwise I agree with you that $(a_2)$ is not satisfied. It does, %however seem to satisfy $(a_1)$. Maybe it needs a clarification or to leave %only $(a_1)$}
%\newline \textcolor{purple}
%{RS: I suggest to replace $\cC$ by $\cC'$ since $(a)$ concerns in Theorem %5.1 the component $\cC'$ and moreover $(b)$ can be formulated as follows $%\cC'$ contains a point $(0,\lambda_\ast)$ with $%\lambda_\ast\in\fB\setminus[\lambda_-,\lambda_+]$} }
\begin{align*}
  \begin{array}{llll}
    1: (a_2) & 2: (a_1) \&(b), &  3: (a_1) \& (a_2),  &  4: (a_1)\& (a_2),\\
    5: (a_2)\& (b), &   6: (a_1) \& (a_2)\& (b), & 7: (b). &  \\
  \end{array}
\end{align*}
\begin{proof}
	Our aim is to apply the abstract \tref{thmglobal} with the following Banach spaces $X:=W_0^{1,\infty}(\R)$, $Y:=L_0^\infty(\R)$ in order to describe the zeros of the operator $G:U\to L_0^\infty(\R)$ from \eqref{gdef}.

    First of all, due to \tref{thmgprop} the domain $U\subseteq W_0^{1,\infty}(\R)\tm\Lambda$ is nonempty, open and simply connected, while $G$ and $D_1G$ exist as continuous functions with $G(0,\lambda)\equiv 0$ on $\Lambda$, so that the assumptions $(M_1$--$M_2)$ in App.~\ref{appC} are fulfilled. From \pref{propfred} results the inclusion $D_1G(x,\lambda)\in F_0(W_0^{1,\infty}(\R),L_0^\infty(\R))$ for all $(x,\lambda)\in U$ and also $(M_3)$ holds. 

	Our assumption $E(\lambda_-)E(\lambda_+)<0$ implies $E(\lambda_-),E(\lambda_+)\neq 0$, \pref{propEvans} yields that the continuous path $T$ defined in \eqref{tpath} satisfies $T(\lambda)\in GL(W_0^{1,\infty}(\R),L_0^\infty(\R))$ for $\lambda\in\set{\lambda_-,\lambda_+}$. Therefore, $T$ has invertible endpoints, the parity $\sigma(T,[\lambda_-,\lambda_+])$  is well-defined (cf.~App.~\ref{appC}) and \cite[Thm.~3.6]{poetzsche:skiba:24} implies
	$$
        \sigma(T,[\lambda_-,\lambda_+])
        =
        \sgn E(\lambda_-)\cdot\sgn E(\lambda_+)=-1. 
    $$
    These preparations yield that \tref{thmglobal} applies to the concrete equation \eqref{abs} as defined in Sect.~\ref{sec2}. As a result, there exists a value $\lambda^\ast\in(\lambda_-,\lambda_+)$ and a maximal connected set $\cC'\subseteq W_0^{1,\infty}(\R)\tm\Lambda$ containing solutions to \eqref{abs} bifurcating in $(0,\lambda^\ast)$ from the trivial branch. Using \tref{thmchar}(b) the solutions to \eqref{abs} translate into bounded entire solutions to the Carath{\'e}odory equations \eqref{cde}. This implies the claim. In particular, (a) and (b) are immediate consequences of the respective alternatives stated in \tref{thmglobal}. 
\end{proof}

One can further specify the set $\cC'$ from the above alternative $(a)$. Here we use the projection $\pi_1:W_0^{1,\infty}(\R)\tm\R\to W_0^{1,\infty}(\R)$ defined in \eqref{noproj}: 
\begin{corollary}[structure of $\cC'$]
	\label{corbifODE}
    The set $\cC'$ in the alternative (a) of \tref{thmbif} is connected and has at least one of the following properties:
%    The set $\cC'$ from the alternative (a) of \tref{thmbif} is connected and its closure $\overline{\cC'}$ in $W_0^{1,\infty}(\R)\tm\R$ has at least one of the following properties:
    \begin{enumerate}
        \item[$(a'_1)$] $\overline{\cC'}\cap\partial U\neq\emptyset$, where $\overline{\cC'}$ denotes the closure in $W_0^{1,\infty}(\R)\tm\R$,

        \item[$(a'_2)$] $\cC'$ is unbounded,

        \item[$(a'_3)$] there exists an $\eps>0$ such that for each partition $\set{P_1,\ldots,P_n}$ of $\R$ there is an $j\in\set{1,\ldots,n}$ such that for any set $N\subseteq P_j$ of measure $0$ there exist $s,t\in P_j\setminus N$ and a function  $x\in\cC'$ satisfying $\abs{\dot x(t)-\dot x(s)}\geq\eps$,

        \item[$(a'_4)$] there exists an $\eps>0$ such that for each $T>0$ and every set $N\subseteq\R\setminus(-T,T)$ of measure $0$ there exists a $t\in\R\setminus((-T,T)\cup N)$ and a function $x\in\cC'$ satisfying $\abs{x(t)}\geq\eps$ or $\abs{\dot x(t)}\geq\eps$.
    \end{enumerate}
%    \begin{enumerate}
%        \item[$(a'_1)$] It is unbounded, 
%
%        \item[$(a'_2)$] $\pi_1\overline{\cC'}$ is not closed (this alternative cannot occur, if $\Lambda$ is bounded), 

%        \item[$(a'_3)$] $\pi_2\overline{\cC'}$ is not closed, 

%        \item[$(a'_4)$] there exists an $\eps>0$ such that for each partition $\set{P_1,\ldots,P_n}$ of $\R$ there is an $j\in\set{1,\ldots,n}$ such that for any set $N\subseteq P_j$ of measure $0$ there exist $s,t\in P_j\setminus N$ and a function  $x\in\pi_1\overline{\cC'}$ satisfying $\abs{\dot x(t)-\dot x(s)}\geq\eps$, 

%        \item[$(a'_5)$] there exists an $\eps>0$ such that for each $T>0$ and every set $N\subseteq\R\setminus(-T,T)$ of measure $0$ there exists a $t\in\R\setminus((-T,T)\cup N)$ and a function $x\in\pi_1\overline{\cC'}$ satisfying $\abs{x(t)}\geq\eps$ or $\abs{\dot x(t)}\geq\eps$. 
%    \end{enumerate}
\end{corollary}
\begin{proof}
    The connectedness of the set $\cC'$ is an immediate consequence of its construction in the proof of \tref{thmbif}.

 Assume  $(a'_2)$, $(a'_3)$ and $(a'_4)$ are not satisfied. Note that, from \cref{corW1comb}, $(a'_2)$, $(a'_3)$ and $(a'_4)$ characterize the options for $\overline{\cC'}$ to be non-compact in $W_0^{1,\infty}(\R)\tm\R$.
 We will show that $\overline{\cC'}$ satisfies $(a'_1)$.
    Indeed, \cref{corW1comb} implies that $\overline{\cC'}$ is compact in $W^{1,\infty}_0(\R)\times\R$, in particular in $\overline{U}$ (the closure of $U$ in $W_0^{1,\infty}(\R)\tm\R$). Since $\cC'$ is non-compact in $U$, it follows that $(\overline{\cC'}\setminus \cC')\cap U=\emptyset$. But $\overline{\cC'}\subset \overline{U}$ with $\cC'\subset U$, and hence $(\overline{\cC'}\setminus \cC')\cap \partial U\neq\emptyset$ since $\overline{U}=U\cup \partial U$.
    %Due to \lref{lemmaA} the alternative \tref{thmbif}(a) is equivalent to the fact that at least one of the images $\pi_1\overline{\cC'}\subseteq W_0^{1,\infty}(\R)$ or $\pi_2\overline{\cC'}\subseteq\overline{\Lambda}$ is not compact. Referring to the logical negation of the conditions from \cref{corW1comb} and $\pi_2\overline{\cC'}\subseteq\R$ this means that at least one of the following conditions holds:
    %\begin{align*}
    %    (i)\,  &\pi_1\overline{\cC'}\text{ is not closed},&
    %    (ii)\, &\pi_1\overline{\cC'}\text{ is not bounded},&
    %    (iii)\,&(a'_4)\text{ holds},\\
    %    (iv)\, & (a'_5)\text{ holds},&
    %    (v)\,  & \pi_2\overline{\cC'}\text{ is not closed}, &
    %    (vi)\, & \pi_2\overline{\cC'}\text{ is not bounded}. 
    %\end{align*}
    %The alternatives $(ii)$ and $(vi)$ obviously mean that $\cC'$ is unbounded, yielding $(a_1')$. Hence, it remains to show that boundedness of $\Lambda$ excludes the alternative $(a'_2)$ because then $\pi_1\overline{\cC'}$ is closed. Let $((x_n,\lambda_n))_{n\in\N}$ be a sequence in $\overline{\cC'}$ such that $(\pi_1(x_n,\lambda_n))_{n\in\N}$ converges to some $x\in W_0^{1,\infty}(\R)$, which is also the limit of $(x_n)_{n\in\N}$. Since $\overline{\Lambda}$ is compact, a subsequence $(\lambda_{k_n})_{n\in\N}$ converges to $\lambda\in\overline{\Lambda}$. Because $\pi_1$ is continuous, this implies the relation $x=\lim_{n\to\infty}\pi_1(x_{k_n},\lambda_{k_n})=\pi_1(x,\lambda)$. Hence, $x\in\pi_1\overline{\cC'}$ and $\pi_1\overline{\cC'}$ is closed. 
\end{proof}

We continue with a tightening of \tref{thmbif}(a).
\begin{theorem}[global bifurcation of solutions homoclinic to $\cT$ under admissibility]
	\label{thmglobal2CDE}
	Let $(H_0$--$H_2)$ hold and $\lambda_-,\lambda_+\in\Lambda$ with $\lambda_-<\lambda_+$. If $E(\lambda_-)E(\lambda_+)<0$ and $\alpha(\lambda)\cup\omega(\lambda)\subseteq\sH(\lambda)$ is admissible for all $\lambda\in\Lambda$, then the alternative $(a_2)$ in \tref{thmbif} simplifies to
	\begin{itemize}
		\item[$(a_2')$] $\cC'$ is unbounded. 
	\end{itemize}
	In particular, for $\Omega=\R^d$ and $\Lambda=\R$ the alternative $(a)$ reduces to $(a_2')$. 
\end{theorem}
\begin{proof}
	We intend to apply the abstract \tref{thmglobal2} in the setting provided in the above proof of \tref{thmglobal}. Thereto, it remains to establish that $G:U\to L_0^\infty(\R)$ is proper on closed and bounded subsets of $U$. Due to our admissibility assumption on $\alpha(\lambda)\cup\omega(\lambda)$, $\lambda\in\Lambda$,  this immediately results from \pref{propprop}. 
\end{proof}

\begin{figure}[h]
	\includegraphics[scale=0.5]{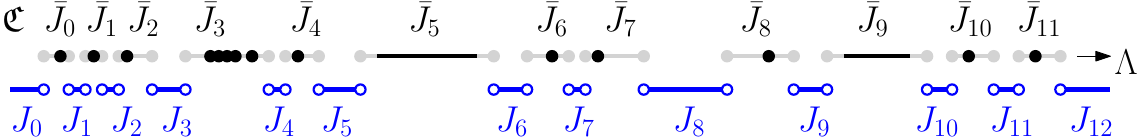}
	\caption{Open intervals $J_0,\ldots,J_{11}$ (blue) such that the family $J$ of compact intervals $\bar J_0,\ldots,\bar J_{10}$ (grey) covers the set $\fC$ of critical values (black) consisting of two intervals contained in $\bar J_5$ resp.\ in $\bar J_9$, an accumulation point contained in $\bar J_3$ and the remaining isolated points.}
	\label{fig3}
\end{figure}
A refinement of \tref{thmbif}(b) requires further preparations:  For this purpose, assume $\I\subseteq\Z$ is a nonempty set of consecutive integers and set $\I':=\set{i\in\I:i+1\in\I}$. Given the set $\fC$ of critical values for \eqref{var}, we assume a family $\set{J_i}_{i\in\I}$ of nonempty, open intervals $J_i\subseteq\R$ has the following properties (cf.~\cite[Lemma~2.1]{lopez:mora:04} and Fig.~\ref{fig3}):
\begin{itemize}
	\item[$(I_1)$] $J_i\cap\fC=\emptyset$ for all $i\in\I$, 

	\item[$(I_2)$] for each $i\in\I'$ one has $x_i<x_{i+1}$ for all $x_i\in J_i$, $x_{i+1}\in J_{i+1}$ such that the compact intervals $\bar J_i:=[\sup J_i,\inf J_{i+1}]$ satisfy $\fC\subseteq\bigcup_{i\in\I'}\bar J_i$; note that $\bar J_i$ might be a singleton, 

	\item[$(I_3)$] $\set{J_i}_{i\in\I}$ is \emph{locally finite}, i.e.\ any $x\in\R$ possesses a neighborhood $I$ so that the set of indices $\set{i\in\I:\,J_i\cap I\neq\emptyset}$ is finite. 
\end{itemize}
% \comment{IPL: In Figure 3, shouldn't the $\bar J_i$ be in fact $\cT_i$?
% \newline\textcolor{red}{RS: I think the picture is correct. Observe that $%\bar J_i$ is contained 
%in $\Lambda$, while $\cT_i$ is contained in $W^{1,%\infty}%(\R)\times\Lambda$}\newline
% IPL OK}
\begin{SCfigure}[2]
	\includegraphics[scale=0.5]{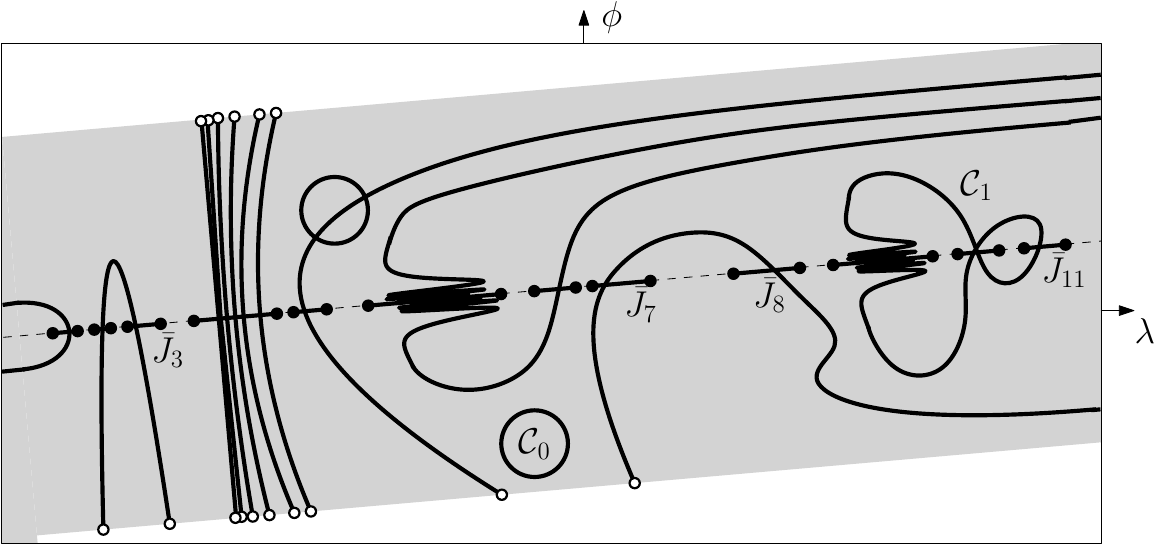}
	\caption{The set $\cS_J$ with $\Lambda=\R$. For the bounded components $\cC_0$ and $\cC_1$ of $\cS$ one obtains 
	\newline $I_J(\cC_0)=0$ and
	\newline
	$I_J(\cC_1)=\set{9,10,11}$
 }
	\label{fig4a}
\end{SCfigure}

We say such a family $J:=\set{\bar J_i}_{i\in\I'}$ of compact intervals \emph{covers} $\fC$ and note that it is locally finite thanks to \cite[Lemma~2.4]{lopez:mora:04}. Furthermore, we define the compact segments
$$
	\cT_i:=\set{(\phi_\lambda,\lambda):\,\lambda\in\bar J_i}
	\subseteq\cT
	\fall i\in\I'
$$
of the prescribed branch $\cT$, as well as the superset of $\cS$ given by (cf.~Fig.~\ref{fig4a}), 
\begin{multline*}
	\cS_J
	:=
	\set{(\psi,\lambda)\in W^{1,\infty}(\R,\Omega)\tm\R:\,\psi\text{ solves }\eqref{cde},\psi\neq\phi_\lambda\text{ and }\psi-\phi_\lambda\in W_0^{1,\infty}(\R)}\\
	\cup
	\bigcup_{i\in\I'}\cT_i, 
\end{multline*}
which is closed in $U$.
The compact sets $\cT_i$ may connect  several components of $\cS$ bifurcating at bifurcation values in $\bar J_i$ to a single component of $\cS_J$ (cf.~Figs.~\ref{fig2a} and~\ref{fig4a}). 

If $\cC\neq\emptyset$ is a connected component of $\cS$, we introduce the set 
$$
	I_J(\cC):=\set{i\in\I':\,\cC\cap\cT_i\neq\emptyset}
$$
of all indices $i$ for which solutions bifurcate from a set $\cT_i$ (cf.~Fig.~\ref{fig4a}). Note that $I_J(\cC)$ is finite for compact $\cC$. For each $i\in\I$ we choose a $\lambda_i\in J_i$ and introduce the $J$-\emph{parity map}
\begin{align*}
	\pi_J:\I'&\to\set{-1,0,1},&
	\pi_J(i)&:=\tfrac{1}{2}(\sgn E(\lambda_{i+1})-\sgn E(\lambda_i)),
\end{align*}
having the properties
\begin{align*}
	\pi_J(i)=-1&\Leftrightarrow E\text{ is positive on $J_i$ and negative on $J_{i+1}$},\\
	\pi_J(i)=0&\Leftrightarrow E\text{ has the same sign on $J_i$ and $J_{i+1}$},\\
	\pi_J(i)=1&\Leftrightarrow E\text{ is negative on $J_i$ and positive on $J_{i+1}$}.
\end{align*}
The values $\pi_J(i)$ hence indicate an (oriented) sign change in a global Evens function $E$, when $\lambda$ increases from the interval $J_i$ to $J_{i+1}$. In particular, $\pi_J(i)=\pm 1$ is sufficient for a bifurcation \cite[Thm.~4.1]{poetzsche:skiba:24}. Finally, summing over consecutive values of $\pi_J$ therefore counts the number of consecutive sign changes of $E(\lambda)$ as the parameter $\lambda$ increases through the parameter interval $\Lambda$ respecting the covering family $J$. 
\begin{theorem}[bounded components]
	\label{thmsignCDE}
	Let $(H_0$--$H_2)$ hold and $\lambda_-,\lambda_+\in\Lambda$, $\lambda_-<\lambda_+$. If $E(\lambda_-)E(\lambda_+)<0$, $\alpha(\lambda)\cup\omega(\lambda)\subseteq\sH(\lambda)$ is admissible for all $\lambda\in\Lambda$ and suppose a family $J$ covers $\fC$. If $\cC\neq\emptyset$ is a bounded component of $\cS_J$, then $\cC$ is compact in $U$ and
	\begin{equation}
		\sum_{i\in I_J(\cC)}\pi_J(i)=0.
		\label{thmsign1a}
	\end{equation}
\end{theorem}
Another interpretation of the condition \eqref{thmsign1a} is as follows: If there exists an $i\in I_J(\cC)$ with $\pi_J(i)=\pm 1$, then there is another index $j\in I_J(\cC)$ such that $\pi_J(j)=\mp 1$. In this case, the continuum $\cC$ connects the sets $\cT_i$ and $\cT_j$. In particular, there always exists an even number of indices $i\in I_J(\cC)$ with $\pi_J(i)\neq 0$. 
\begin{proof}
	As in the proofs of Thms.~\ref{thmbif} and \ref{thmglobal2CDE} above, one shows that the assumptions $(M_1$--$M_3)$ resp.\ the properness (on bounded and closed subsets) of $G:U\to L_0^\infty(\R)$ hold. Adjusting to the notation of App.~\ref{appC} we set $a_i:=\sgn E(\lambda_i)$ for $i\in\I$ and obtain that the resulting $J$-parity reads as $\pi_J(i)=\tfrac{1}{2}(\sgn E(\lambda_{i+1})-\sgn E(\lambda_i))=\tfrac{1}{2}(a_{i+1}-a_i)$ for each $i\in\I'$. Now due to \cite[Thm.~3.6]{poetzsche:skiba:24} the parity $\sigma$ of the path $T$ defined in \eqref{tpath} can be expressed using a global Evans function as
	$$
		\sigma(T,[\lambda_i,\lambda_{i+1}])=\sgn(E(\lambda_i)E(\lambda_{i+1}))\fall i\in\I'
	$$
	and consequently \tref{thmsign} implies the assertion. 
\end{proof}

\begin{corollary}
	\label{corsignCDE}
	Let $\Omega=\R^d$, $\Lambda=\R$, assume $\fC$ is a discrete set and $\lambda_-<\lambda_+$ are reals with $E(\lambda_-)E(\lambda_+)<0$. If $\cC$ denotes a continuum emanating from the segment $\set{(\phi_\lambda,\lambda):\,\lambda\in(\lambda_-,\lambda_+)}$ and one of the conditions
	\begin{itemize}
		\item[(i)] $\cC\cap\set{(\phi_\lambda,\lambda):\,\lambda\in\fC}$ is infinite, 

		\item[(ii)] $\sum_{i\in I_J(\cC)}\pi_J(i)\neq 0$
	\end{itemize}
	holds, then $\cC$ is unbounded. 
\end{corollary}
\begin{proof}
	Based on arguments from the proofs of Thms.~\ref{thmbif} and \ref{thmglobal2CDE} this readily follows from the abstract \cref{corsign}. 
\end{proof}
\section{Illustrations}
\label{sec6}
Let us illustrate the previous Sec.~\ref{sec5} by means of a simple parametrized Carath{\'e}odory equation. The subsequent result guarantees that its bifurcation behavior can be understood on basis of a simple algebraic equation in $\R$.
\begin{lemma}\label{lemsol}
	Let $\beta,\gamma\in\R$ and $n\in\N$. The solution of the planar Carath{\'e}odory equation
	\begin{equation}
		\dot x
		=
		\begin{pmatrix}
			-\sgn t & 0\\
			\gamma & \sgn t
		\end{pmatrix}x
		+
		\begin{pmatrix}
			0\\
			\beta x_1^n
		\end{pmatrix}
		\label{lemsol1}
	\end{equation}
	satisfying the initial condition $x(0)=\xi$ is given by
	$$
		\varphi(t;\xi)
		=
		\begin{pmatrix}
			e^{-\abs{t}}\xi_1\\
			\intcc{\xi_2+\sgn t\intoo{\tfrac{\gamma}{2}\xi_1+\tfrac{\beta}{n+1}\xi_1^n}}e^{\abs{t}}
			-
			\sgn t\intoo{\tfrac{\gamma}{2}\xi_1^2e^{-t}+\tfrac{\beta}{n+1}\xi_1^ne^{-n|t|}}
		\end{pmatrix}
	$$
	for all $t\in\R$ and initial values $\xi\in\R^2$. Furthermore, the following holds:
	\begin{equation}
		\varphi(\cdot;\xi)\in L_0^\infty(\R)
		\quad\Leftrightarrow\quad
		\xi_1\intoo{\gamma+\tfrac{2\beta}{n+1}\xi_1^{n-1}}=0
		\text{ and }\xi_2=0.
		\label{lemsol3}
	\end{equation}
\end{lemma}
\begin{proof}
	Let $\xi\in\R^2$. The first equation in the system \eqref{lemsol1} can be solved independently yielding $\varphi_1(t;\xi)=e^{-\abs{t}}\xi_1$. If we insert this into the second equation of \eqref{lemsol1}, then the variation of constants formula implies the stated expression for $\varphi_2(t;\xi)$. From this 
	\begin{align*}
		\lim_{t\to\infty}\varphi(t;\xi)&=0
		\quad\Leftrightarrow\quad
		\xi_2=-\tfrac{\gamma}{2}\xi_1-\tfrac{\beta}{n+1}\xi_1^n,\\
		\lim_{t\to-\infty}\varphi(t;\xi)&=0
		\quad\Leftrightarrow\quad
		\xi_2=\tfrac{\gamma}{2}\xi_1+\tfrac{\beta}{n+1}\xi_1^n
	\end{align*}
	hold, which leads to the claimed equivalence \eqref{lemsol3}. 
\end{proof}

\begin{example}\label{exbif}
	Let $\gamma:\Lambda\to\R$ be a continuous function and assume $\beta\in\R\setminus\set{0}$, $n\in\N\setminus\set{1}$ are fixed. The planar Carath{\'e}odory equation
	\begin{equation}
		\dot x
		=
		\begin{pmatrix}
			-\sgn t & 0\\
			\gamma(\lambda) & \sgn t
		\end{pmatrix}x
		+
		\begin{pmatrix}
			0\\
			\beta x_1^n
		\end{pmatrix}
		\label{exsol1}
	\end{equation}
	possesses the continuous trivial branch $\phi_\lambda(t):\equiv 0$ of bounded entire solution, i.e.\ one has $\cT=\set{0}\tm\Lambda\subseteq W^{1,\infty}(\R)\tm\R$. On the one hand, using \lref{lemsol} we are able to locate the nontrivial solutions of \eqref{exsol1} being homoclinic to $\phi_\lambda$ by means of the equations
	\begin{align*}
		\xi_1^{n-1}&=-\tfrac{n+1}{2\beta}\gamma(\lambda),&
		\xi_2&=0
	\end{align*}
	determining their corresponding initial condition $x(0)=\xi\in\R^2$ (cf.~\eqref{lemsol3}). On the other hand, linearizing \eqref{exsol1} along the branch $\phi_\lambda$ yields lower triangular variational equations
	$$
		\dot x
		=
		\begin{pmatrix}
			-\sgn t & 0\\
			\gamma(\lambda) & \sgn t
		\end{pmatrix}x
	$$
	having exponential dichotomies on both $\R_+$ and $\R_-$, whose projections satisfy
	\begin{align*}
		R(P_\lambda^+(0))&=\R\binom{-2}{\gamma(\lambda)},&
		N(P_\lambda^-(0))&=\R\binom{2}{\gamma(\lambda)}\fall\lambda\in\Lambda.
	\end{align*}
 Therefore, the global Evans function 
	$$
		E(\lambda)
		=
		\det\begin{pmatrix}
			-2 & 2\\
			\gamma(\lambda) & \gamma(\lambda)
		\end{pmatrix}
		=
		-4\gamma(\lambda)\fall\lambda\in\Lambda
	$$
	first provides the critical values $\fC=\gamma^{-1}(0)$ due to \pref{propevans} and second a sufficient condition for bifurcations in \eqref{exsol1} in terms of sign changes of the coefficient function $\gamma$ according to \cite[Thm.~4.2]{poetzsche:skiba:24}. This criterion is more general than e.g.\ the local bifurcation condition \cite[Thm.~4.1]{poetzsche:12}, which requires $\gamma$ to be of class $C^2$, since we are actually in the scope of \tref{thmbif}. Beyond that, because the right-hand side of \eqref{exsol1} is piecewise autonomous, the limit sets $\alpha(\lambda)$ and $\omega(\lambda)$ of the Bebutov flow are singletons containing the respective functions
    \begin{align*}
        f^-(x,\lambda)&:=
		\begin{pmatrix}
			1 & 0\\
			\gamma(\lambda) & -1
		\end{pmatrix}x
		+
		\begin{pmatrix}
			0\\
			\beta x_1^n
		\end{pmatrix},&
        f^+(x,\lambda)&:=
		\begin{pmatrix}
			-1 & 0\\
			\gamma(\lambda) & 1
		\end{pmatrix}x
		+
		\begin{pmatrix}
			0\\
			\beta x_1^n
		\end{pmatrix}.
    \end{align*}
    From $f^-(x,\lambda)=0\Leftrightarrow x=0$ and $f^+(x,\lambda)=0\Leftrightarrow x=0$ one obtains the compact and totally disconnected set $Z_{\alpha(\lambda)\cup\omega(\lambda)}=\set{0}$. Moreover, the autonomous ordinary differential equations $\dot x=f^-(x,\lambda)$ and $\dot x=f^+(x,\lambda)$ induce the respective flows
    \begin{align*}
        \varphi_\lambda^-(t;\xi)&=
        \begin{pmatrix}
            e^t\xi_1\\
            e^{-t}\xi_2-\gamma(\lambda) te^{-t}\xi_1+\beta\frac{e^{-t}-e^{-nt}}{n-1}\xi_1^n
        \end{pmatrix},\\
        \varphi_\lambda^+(t;\xi)&=
        \begin{pmatrix}
            e^{-t}\xi_1\\
            e^t\xi_2-\gamma(\lambda) te^t\xi_1-\beta\frac{e^t-e^{nt}}{n-1}\xi_1^n
        \end{pmatrix},
    \end{align*}
    which both are bounded on $\R$, precisely for the initial value $\xi=0$. In conclusion, the union $\alpha(\lambda)\cup\omega(\lambda)$ is admissible for all $\lambda\in\Lambda$. For this reason, the more specific Thms.~\ref{thmglobal2CDE}, \ref{thmsignCDE} and \cref{corsignCDE} apply. We discuss several examples: 

	(1) For $\Lambda=\R$, $\gamma(\lambda):=\lambda$ the critical values $\fC=\set{0}$ follow. The global Evans function $E(\lambda)=-4\lambda$ changes sign at $0$ and hence \tref{thmbif} guarantees a bifurcation at $\lambda^\ast=0$. In fact $(0,0)$ is the only bifurcation point for equation \eqref{exsol1}, $\fB=\fC=\set{0}$, and more precisely, one has (cf.\ Fig.~\ref{figbif1}):
	\begin{figure}[ht]
		\includegraphics[scale=0.59]{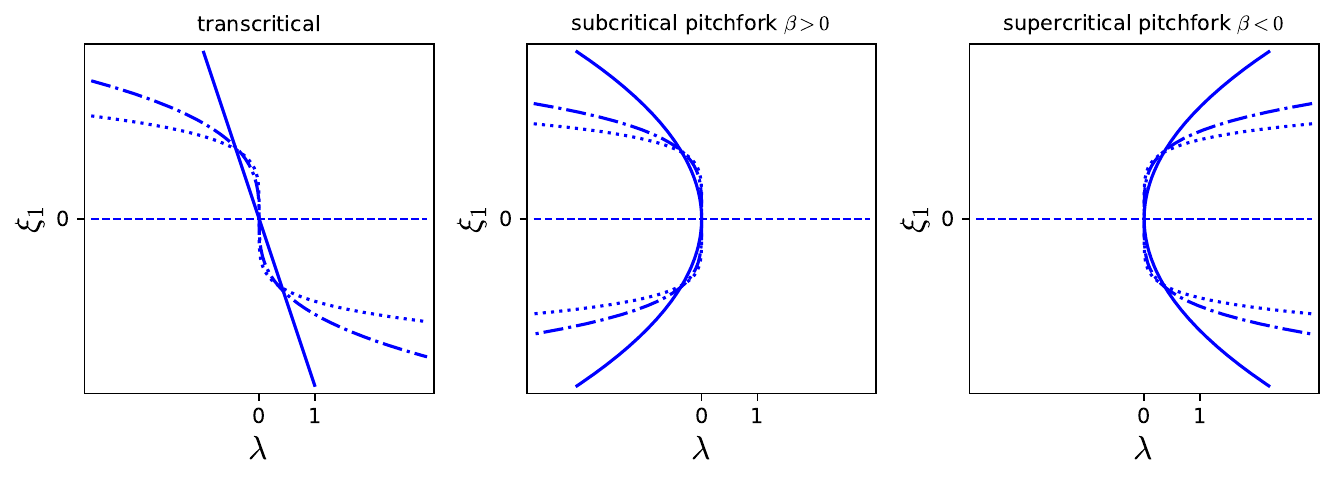}
		\caption{Bifurcating branches for $\gamma(\lambda):=\lambda$ in \eref{exbif}(1): 
		\newline
		left: Transcritical bifurcation for $n=2$ (solid), $n=4$ (dashed), $n=6$ (dotted) and $\beta>0$
		\newline
		center and right: Pitchfork bifurcations for $n=3$ (solid), $n=5$ (dashed), $n=7$ (dotted)}
		\label{figbif1}
	\end{figure}
	\begin{itemize}
		\item If $n$ is even, then there is a transcritical bifurcation with the emanating branch 
		$$
			\cC=\cC'=\set{\intoo{\varphi\intoo{\cdot;\sqrt[n-1]{-\tfrac{n+1}{2\beta}\lambda},0},\lambda}\mid\lambda\in\R}
		$$
        of entire solutions of \eqref{exsol1} homoclinic to $0$. 

		\item If $n$ is odd, then a pitchfork bifurcation occurs, which is subcritical for $\beta>0$ with emanating branch
		$$
			\cC=\cC'=\set{\intoo{\varphi\intoo{\cdot;\pm\sqrt[n-1]{-\tfrac{n+1}{2\beta}\lambda},0},\lambda}\mid\lambda\leq 0}
		$$
		and supercritical for $\beta<0$ with emanating branch
		$$
			\cC=\cC'=\set{\intoo{\varphi\intoo{\cdot;\pm\sqrt[n-1]{-\tfrac{n+1}{2\beta}\lambda},0},\lambda}\mid\lambda\geq 0}. 
		$$
	\end{itemize}
	Throughout, the branches of nontrivial homoclinic solutions are unbounded in the $x$- and $\lambda$-direction. Thus, the alternative $(a_2')$ of \tref{thmglobal2CDE} and \cref{corbifODE}$(a_2')$ is covered. 

	(2) For the continuous function $\gamma(\lambda):=\abs{\lambda}$ on $\Lambda=\R$ we also obtain $\fC=\set{0}$ as critical values. Now the Evans function $E(\lambda)=-4\abs{\lambda}$ has a zero at $0$, but does not change sign and neither \tref{thmbif} nor \ref{thmglobal2CDE} do apply. Nevertheless, nontrivial homoclinic solutions might bifurcate from $\cT$ in $(0,0)$ (cf.\ Fig.~\ref{figbif2}):
	\begin{figure}[ht]
		\includegraphics[scale=0.59]{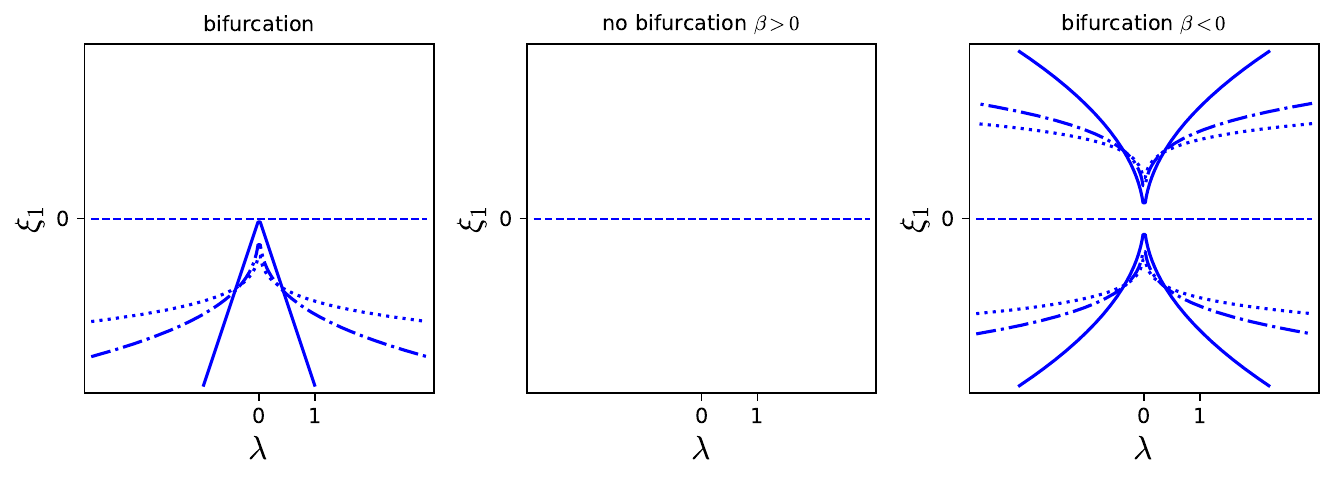}
		\caption{Bifurcating branches for $\gamma(\lambda):=\abs{\lambda}$ in \eref{exbif}(2):
		\newline
		left: Bifurcation for $n=2$ (solid), $n=4$ (dashed), $n=6$ (dotted) and $\beta>0$
		\newline
		center and right: No bifurcation resp.\ bifurcation for $n=3$ (solid), $n=5$ (dashed), $n=7$ (dotted)}
		\label{figbif2}
	\end{figure}
	\begin{itemize}
		\item If $n$ is even, then the nontrivial branch is unbounded and reads as
		$$
			\cC=\cC'=\set{\intoo{\varphi\intoo{\cdot;\sqrt[n-1]{-\tfrac{n+1}{2\beta}\abs{\lambda}},0},\lambda}\mid\lambda\in\R}.
		$$

		\item If $n$ is odd, then there exists no nontrival bifurcating branch for $\beta>0$. For $\beta<0$ however, the set
		$$
			\cC=\cC'=\set{\intoo{\varphi\intoo{\cdot;\pm\sqrt[n-1]{-\tfrac{n+1}{2\beta}\abs{\lambda}},0},\lambda}\mid\lambda\in\R}
		$$
		is the union of two  nontrivial and unbounded branches intersecting at $(0,\lambda)$. 
	\end{itemize}
	This demonstrates that a sign change in an Evans function is a sufficient (for this, see \cite[Thm.~4.2]{poetzsche:skiba:24}), but not a necessary bifurcation criterion. 

	(3) For the $2\pi$-periodic function $\gamma(\lambda):=\sin\lambda$ on $\Lambda=\R$ countably many critical values $\fC:=\set{\mu_i:\,i\in\Z}$ with $\mu_i:=\pi i$ follow. Because the global Evans function $E(\lambda)=-4\sin\lambda$ changes sign at each $\mu_i$, $i\in\Z$, \tref{thmbif} implies bifurcations and thus $\fB=\fC=\set{\mu_i:\,i\in\Z}$. More detailed (cf.\ Fig.~\ref{figbif3}): 
	\begin{figure}[ht]
		\includegraphics[scale=0.59]{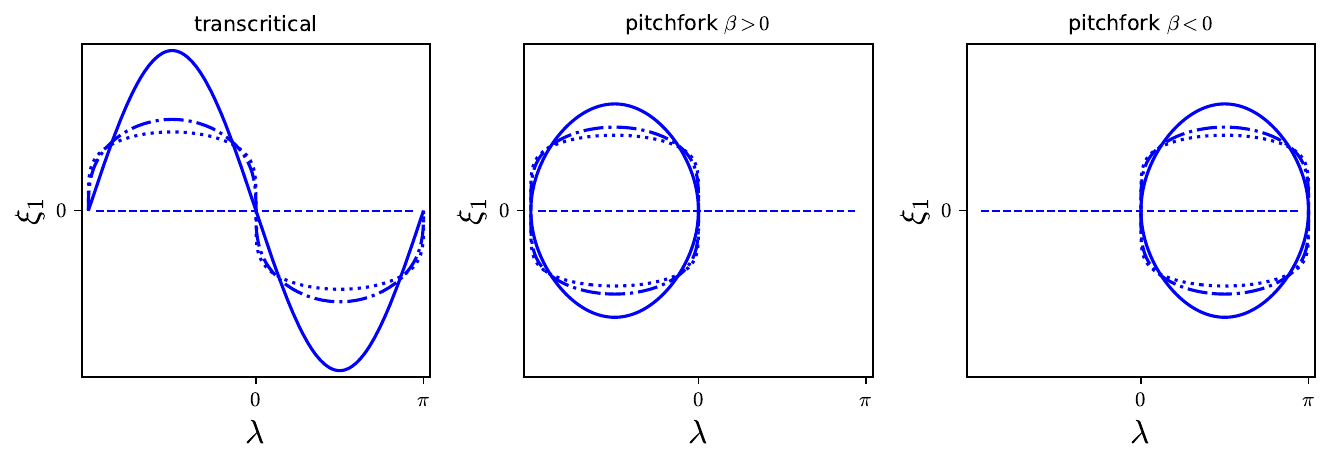}
		\caption{Bifurcating branches for $\gamma(\lambda):=\sin\lambda$ in \eref{exbif}(3): 
		\newline
		left: Transcritical bifurcation for $n=2$ (solid), $n=4$ (dashed), $n=6$ (dotted)
		\newline
		center and right: Pitchfork bifurcations for $n=3$ (solid), $n=5$ (dashed), $n=7$ (dotted)}
		\label{figbif3}
	\end{figure}
	\begin{itemize}
		\item If $n$ is even, then $\cC$ is an unbounded branch and $2\pi$-periodic in $\lambda$ given by
		$$
			\cC=\cC'=
			\set{\intoo{\varphi\intoo{\cdot;\sqrt[n-1]{-\tfrac{n+1}{2\beta}\sin\lambda},0},\lambda}:\lambda\in\R}, 
		$$ 
		which bifurcates transcritically from the trivial branch $\cT$ at each $\mu_i$. We are in the situation of \tref{thmglobal2CDE}$(a_2')$ and \cref{corbifODE}$(a_2')$. The fact that $\cC$ is unbounded also results from \cref{corsignCDE}(i) since $\cC\cap\set{(0,\mu_i):\,i\in\Z}=\set{(0,\mu_i):\,i\in\Z}$ is infinite. 

		\item If $n$ is odd, then there are pitchfork bifurcations along the trivial branch $\cT$ at each $\mu_i$, $i\in\Z$. For $\beta>0$ they are subcritical in $\mu_{2i}$ and supercritical in $\mu_{2i-1}$, while each bounded branch
		$$
			\cC_i=
			\set{\intoo{\varphi\intoo{\cdot;\pm\sqrt[n-1]{-\tfrac{n+1}{2\beta}\sin\lambda},0},\lambda}\mid\mu_{2i-1}\leq\lambda\leq \mu_{2i}}
		$$
		connects $(0,\mu_{2i-1})$ with $(0,\mu_{2i})$ for all $i\in\Z$. For $\beta<0$ the branches bifurcate supercritically in $\mu_{2i}$ and subcritically in $\mu_{2i+1}$, while each bounded branch
		$$
			\cC_i=
			\set{\intoo{\varphi\intoo{\cdot;\pm\sqrt[n-1]{-\tfrac{n+1}{2\beta}\sin\lambda},0},\lambda}\mid \mu_{2i}\leq\lambda\leq\mu_{2i+1}}
		$$
		connects the points $(0,\mu_{2i})$ and $(0,\mu_{2i+1})$ for all $i\in\Z$. In both cases these  branches $\cC_i$ return to $\cT$, which exemplifies \tref{thmbif}(b). Indeed, \tref{thmsignCDE} implies that the bounded sets $\cC_i$ are compact. More detailed, in order to mimic the framework of \tref{thmsignCDE} we choose $\I=\Z$, intervals $\bar J_i=\set{\mu_i}$, $J_i=(\mu_{i-1},\mu_i)$ and $\lambda_i:=\tfrac{1}{2}(\mu_{i-1}+\mu_i)$, $i\in\Z$. For $\beta>0$ it is $I_J(\cC_i)=\set{2i-1,2i}$,
		$$
			\sum_{j\in I_J(\cC_i)}\pi_J(j)
			=
			\pi_J(2i)+\pi_j(2i-1)
			=
			\tfrac{1}{2}(\sgn E(\lambda_{2i})-\sgn E(\lambda_{2i-2}))
			=
			0,
		$$
		while for $\beta<0$ results $I_J(\cC_i)=\set{2i,2i+1}$, 
		$$
			\sum_{j\in I_J(\cC_i)}\pi_J(j)
			=
			\pi_J(2i+1)+\pi_j(2i)
			=
			\tfrac{1}{2}(\sgn E(\lambda_{2i+2})-\sgn E(\lambda_{2i}))
			=
			0,
		$$
		which confirms the relation \eqref{thmsign1a}. 
	\end{itemize}

	(4) The continuous function $\gamma(\lambda):=\tan\lambda$ on the open interval $\Lambda=(-\tfrac{\pi}{2},\tfrac{\pi}{2})$ yields a single critical value $\fC=\set{0}$. We are in the framework of \tref{thmbif}, \cref{corbifODE} and the bifurcation scenario resembles the cases (1) and (3) locally: The global branches of the transcritical bifurcation connect the hyperplanes $W_0^{1,\infty}(\R)\tm\set{-\tfrac{\pi}{2}}$ and $W_0^{1,\infty}(\R)\tm\set{\tfrac{\pi}{2}}$ (cf.\ Fig.~\ref{figbif4}(left)). The global branches of the pitchfork bifurcation are asymptotic to $W_0^{1,\infty}(\R)\tm\set{-\tfrac{\pi}{2}}$ in the subcritical case (see Fig.~\ref{figbif4}(center)) resp.\ to $W_0^{1,\infty}(\R)\tm\set{\tfrac{\pi}{2}}$ in the supercritical case (see Fig.~\ref{figbif4}(right)); this illustrates \tref{thmbif}(a) and more specifically \tref{thmglobal2CDE}$(a_2')$ resp.\ \cref{corbifODE}$(a_2')$. In any case, the bifurcating branches $\cC$ are unbounded in $x$-direction. 
	\begin{figure}[ht]
		\includegraphics[scale=0.59]{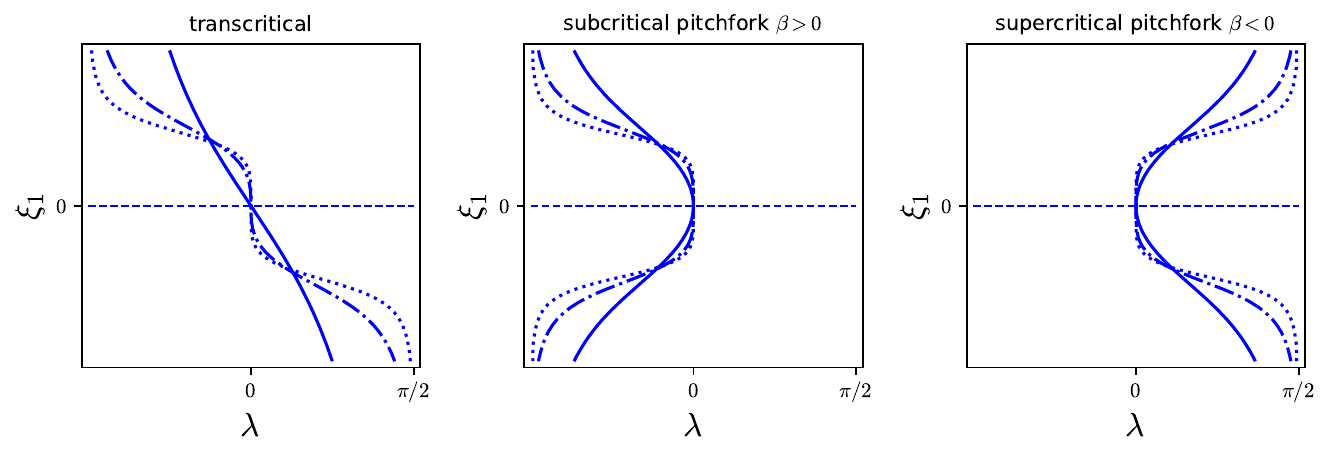}
		\caption{Bifurcating branches for $\gamma(\lambda):=\tan\lambda$ in \eref{exbif}(4): 
		\newline
		left: Transcritical bifurcation for $n=2$ (solid), $n=4$ (dashed), $n=6$ (dotted)
		\newline
		center and right: Pitchfork bifurcations for $n=3$ (solid), $n=5$ (dashed), $n=7$ (dotted)}
		\label{figbif4}
	\end{figure}
%
%	(5) The situation $\gamma(\lambda):=\sin\tfrac{1}{\lambda}$ on $\Lambda=(0,\infty)$ with critical values $\fC=\set{\tfrac{1}{\pi i}:\,i\in\N}$ is illustrated in Fig.~\ref{figbif5} and can be analyzed analogously to the above case (3) after the parameter transformation $\lambda\mapsto\tfrac{1}{\lambda}$. 
%	\begin{figure}[ht]
%		\includegraphics[scale=0.36]{bifurcation05a.png}
%		\includegraphics[scale=0.36]{bifurcation05b.png}
%		\includegraphics[scale=0.36]{bifurcation05c.png}
%		\caption{Bifurcating branches for $\gamma(\lambda):=\sin\tfrac{1}{\lambda}$ in \eref{exbif}(5): 
%		\newline
%		left: Transcritical bifurcation for $n=2$ (solid), $n=4$ (dashed), $n=6$ (dotted)
%		\newline
%		center and right: Pitchfork bifurcations for $n=3$ (solid), $n=5$ (dashed), $n=7$ (dotted)}
%		\label{figbif5}
%	\end{figure}
\end{example}
\section*{Acknowledgements}
\noindent
Essential parts of this work were done during a Research-In-Groups stay in Edinburgh, Scotland, UK funded by the \emph{International Centre for Mathematical Sciences}. The generous financial support from that institution is gratefully acknowledged. 
\appendix
\renewcommand{\theequation}{\Alph{section}.\arabic{equation}}
\renewcommand{\thesection}{\Alph{section}}
\section*{Appendices}
\noindent
For the convenience of the reader we briefly sketch the central  constructions required in the above text. They are based on \cite{thesis:Longo} resp.\ \cite{fitzPeja86, FiPejsachowiczV, FitPejRab,fitzpatrick:etal:92} and \cite{lopez:mora:04,lopez:sampedro:22,lopez:sampedro:24}. 
\section{Topological dynamics of Carath\'eodory functions}
\label{appA}
Let $\Omega\subseteq\R^d$ be nonempty, open. A \emph{Lipschitz Carath\'eodory function} $f\colon\R\times\Omega\to\R^d$, in short $f\in \LC$, is defined by the properties
\begin{enumerate}
    \item[$(\mathbf{C})$] $f$ is measurable and for every compact set $K\subset\Omega$ there exists a real-valued function $m_K\in L^1_{loc}(\R)$, called $m$-\emph{bound} in the following, such that for a.a.~$t\in\R$ the boundedness condition $|f(t,x)|\le m_K(t)$ holds for all $x\in K$,
%    \comment{CP: I combined the former conditions $C_1$ and $C_2$ because they only occur in combination\newline \textcolor{red}{RS: OK}}

    \item[$(\mathbf{L})$] for every compact set $K\subset\Omega$ there exists a real-valued function $l_K\in L^1_{loc}(\R)$, called $l$-\emph{bound} in the following, such that for a.a.~$t\in\R$ the Lipschitz condition $|f(t,x)-f(t,y)|\le l_K(t)|x-y|$ holds for all $x,y\in K$. 
\end{enumerate}
Furthermore, a function $f\colon\R\times\Omega\to \R^d$ is said to be \emph{strong Carath\'eodory}, in symbols $f\in \SC$, if it satisfies $(C)$ and
\begin{enumerate}
    \item[$(\mathbf{S})$] for almost every $t\in\R$, the function $f(t,\cdot):\Omega\to\R^d$ is continuous.
\end{enumerate}
In these definitions of $\LC$ and $\SC$, functions that for a.a.~$t\in\R$ coincide for each $x\in\Omega$ are identified. A function $f\in\LC$ is said to have \emph{essentially bounded $m$}- or \emph{$l$-bounds} if the inequalities in $(C)$ and $(L)$ hold with real constants $m_K,l_K>0$ for all $K\subset\R^d$ compact. This is the setting of our work (see Remark \ref{rmk:Carat-assump}) and the results in this appendix are therefore written under this assumption although they originally refer to weaker notions known as $L^1_{loc}$-boundedness and $L^1_{loc}$-equicontinuity of the $m$- or $l$-bounds (cf.~Defs.~2.16, 2.17 and 2.24 in \cite{thesis:Longo}).  
As follows we introduce some fundamental notions in Topological Dynamics such as the hull of a function and the Bebutov flow. We keep the presentation as self contained as possible and refer the reader interested in further detail to \cite{sell:71}. 

The vector space $\LC$ together with the countable family of seminorms
\[
     n_{[p,q],x}(f)
     :=
     \abs{\int_{p}^{q}f(t,x)\d t}\quad
 \text{for $p,\,q\in\Q$, $x\in\Omega\cap\Q^d$ with $\;p<q$}, \text{ and }
 f\in\LC,\]
is a locally convex metric space with 
\begin{equation}\label{eq:dist-sigmaD}
    d(f,g):=\sum_{i,j\in\N}\,\frac{1}{2^{i+j}}\;\min\set{1,n_{[p_i,q_i],x_j}(f-g)}\,,
\end{equation}
where $((p_i,q_i))_{i\in\N}$ is a sequence in $\Q^2$ which is dense in $\R^2$ and satisfies $p_i<q_i$ for all $i\in\N$, and $(x_j)_{j\in\N}$ is a sequence in $\Omega\cap\Q^d$ which is dense in $\Omega$. We shall denote the topology induced by the above distance by $\sigma_\Q$. 

Note that if $f\in\LC$, then, for any $t\in\R$, its time-translation $S^tf:\R\times\Omega\to\R^d$, defined by $S^tf(s,x):=f(t+s,x)$, also belongs to $\LC$. We call  \emph{the hull of $f$ in $(\LC,\sigma_\Q)$}, the metric subspace of  $(\LC,\sigma_\Q)$ defined~by
\begin{equation*}
	\sH_{(\LC,\sigma_\Q)}(f):=\mathrm{cls}_{(\LC,\sigma_\Q)}\left\{S^tf\mid t\in\R\right\},
\end{equation*}
where, $\mathrm{cls}_{(\LC,\sigma_\Q)}(A)$ represents the closure in $(\LC,\sigma_\Q)$ of a set $A$ and $\sH_{(\LC,\sigma_\Q)}(f)$ is endowed with the topology induced by $\sigma_\Q$. Moreover, if $f\in\LC$ has essentially bounded $m$- or $l$-bounds, the same holds for any  $g\in\sH_{(\LC,\sigma_\Q)}(f)$ (cf.~\cite[Prop.~2.26]{thesis:Longo}).

It must be noted that, unlike in the case of bounded and uniformly continuous functions equipped with the compact-open topology, $\mathrm{cls}_{(\LC,\sigma_\Q)}(A)$ is generally not compact. However, this is true under certain assumptions on the $m$- and $l$-bounds for $f$ (which incidentally are always satisfied provided that the $m$-bounds and $l$-bounds are essentially bounded). 
The original version of this result is due to Artstein \cite[Prop.~2.4]{paper:ZA1} with respect to a topology  which is in principle stronger than $\sigma_\Q$. 
Nonetheless, it also holds true for $\sigma_\Q$ \cite[Thm.~2.39]{thesis:Longo}. 
Furthermore, the map
\[
	S:\R\times\sH_{(\LC,\sigma_\Q)}(f)\to\sH_{(\LC,\sigma_\Q)}(f),\quad (t,g)\mapsto S(t,g)=S^tg, 
\]
defines a continuous flow on $\sH_{(\LC,\sigma_\Q)}(f)$ called the \emph{Bebutov flow} \cite{paper:ZA1, thesis:Longo, paper:LNO2}. As an immediate consequence, tools from Topological Dynamics \cite{sell:71} become available. In this regard, the following sets are important. The $\omega$-\emph{limit set} of $f$ is defined as 
\begin{equation*}
	\omega(f)
	:=
	\left\{g\in\sH_{(\LC,\sigma_\Q)}(f)\,\big|\,\exists s_n\to\infty:S^{s_n}f\xrightarrow{\sigma_\Q}g\right\},
\end{equation*}
while the $\alpha$-\emph{limit set} of $f$ is defined as 
\begin{equation*}
	\alpha(f)
	:=
	\left\{g\in\sH_{(\LC,\sigma_\Q)}(f)\,\big|\,\exists s_n\to-\infty:S^{s_n}f\xrightarrow{\sigma_\Q}g\right\},
\end{equation*}
where the notation $S^{s_n}f\xrightarrow{\sigma_\Q}g$ means that $(S^{s_n}f)_{n\in\N}$ converges to $g$ with respect to $\sigma_\Q$ as $n\to\infty$. We gather all the previous information in the following lemma.

\begin{lemma} \label{lem:Top-dyn-Hsigma}
    If $f\in\LC$ has essentially bounded $m$-bounds and $l$-bounds, then the following statements are true: 
\begin{enumerate}
    \item[(a)] Any function $g\in\sH_{(\LC,\sigma_\Q)}(f)$ shares the same essentially bounded $m$-bounds and $l$-bounds with $f$.

    \item[(b)] $\sH_{(\LC,\sigma_\Q)}(f)$ is compact in $(\LC,\sigma_\Q)$. In particular, the limit sets $\alpha(g),\omega(g)$ are nonempty and compact for all $g\in\sH_{(\LC,\sigma_\Q)}(f)$.

    \item[(c)] For every real sequence $(s_n)_{n\in\N}$ with $ |s_n|\to\infty$ there exists a $g\in\sH_{(\LC,\sigma_\Q)}(f)$ and a subsequence $(s_{n_k})_{k\in\N}$ so that $S^{s_{n_k}}f\xrightarrow{\sigma_\Q}g$ as $k\to\infty$.
	\end{enumerate}
\end{lemma}

The topology $\sigma_\Q$ is defined pointwise and thus has practical advantages in terms of taking limits. However, it can be expected that it is in general not sufficiently strong to guarantee convergence of the integrals when continuous functions are plugged in place of the space variable (which we require when dealing with solutions of Carath\'eodory equations). To this end, we  consider a stronger topology requiring uniform convergence of the integrals on certain compact subsets of continuous functions. We call  \emph{a suitable set of moduli of continuity}, any countable  set of non-decreasing continuous functions
\begin{equation*}
	\Theta=\left\{\theta^I_j \in C(\R_+, \R_+)\mid j\in\N, \ I=[q_1,q_2], \ q_1,q_2\in\Q\right\}
\end{equation*}
such that $\theta^I_j(0)=0$ for every $\theta^I_j\in\Theta$, and  with the relation of partial order given~by
\begin{equation*}\label{def:modCont}
	\theta^{I_1}_{j_1}\le\theta^{I_2}_{j_2}\quad \text{whenever } I_1\subseteq I_2 \text{ and } j_1\le j_2 \, .
\end{equation*}
We call $\sigma_{\Theta}$ the topology on $\LC$ generated by the countable family of seminorms \vspace{-.025cm}
\begin{equation*}
	p_{I,\, j}(f)=\sup_{x(\cdot)\in\sK_j^I}\left|\,\int_If\big(t,x(t)\big)\d t\,\right| ,\quad f\in\LC,
\vspace{-.025cm}
\end{equation*}
with $I=[q_1,q_2]$, $q_1,q_2\in\Q$, $j\in\N$, and $\sK_j^I$ is the compact set of continuous functions $x:I\to\bar B_j(0)$ admitting $\theta^I_j$ as a modulus of continuity. The topological space $\left(\LC,\sigma_{\Theta}\right)$ is locally convex and metrizable in analogy to \eqref{eq:dist-sigmaD}. The following result, adapted from \cite[Thm.~2.33]{thesis:Longo}, guarantees that on the hull of $f$ the topology $\sigma_\Q$ and any topology $\sigma_\Theta$ defined as above coincide, provided that $f$ has essentially bounded $l$-bounds.

\begin{lemma}\label{lem:equiv-top} 
    If $f\in\LC$ has essentially bounded $l$-bounds, then for any suitable set of moduli of continuity $\Theta$ one has
$\sH_{(\LC,\sigma_\Q)}(f)=\sH_{(\LC,\sigma_\Theta)}(f)$.     
\end{lemma}
\section{Compactness and properness}
\label{appB}
Let $X,Y$ and $\Lambda$ be metric spaces. With the projections
\begin{align}
	\pi_1:X\tm\Lambda&\to X,&
	\pi_1(x,\lambda)&:=x,&
	\pi_2:X\tm\Lambda&\to\Lambda,&
	\pi_2(x,\lambda):=\lambda
	\label{noproj}
\end{align}
the following is immediate:
\begin{lemma}\label{lemmaA}
    For closed sets $\cC\subseteq X\tm\Lambda$ the following are equivalent:
    \begin{enumerate}[(a)]
        \item $\cC$ is compact, 

        \item there exist compact $K_1\subseteq X$ and $K_2\subseteq\Lambda$ such that $\cC\subseteq K_1\tm K_2$, 

        \item $\pi_1\cC$ and $\pi_2\cC$ are compact.
    \end{enumerate}
\end{lemma}
%\begin{proof}
%    $(a)\Rightarrow(c)$ Because the projections $\pi_1,\pi_2$ are continuous, $\pi_1\cC\subseteq X$ and $\pi_2\cC\subseteq\Lambda$ are compact as images of the compact set $\cC$. 
%    %Moreover, due to $\cC_\lambda:=\set{x\in X:\,(x,\lambda)\in\cC}=\cC\cap(X\tm\set{\lambda})$ for each $\lambda\in\Lambda$ the fibers are closed as intersection of closed sets. 
%
%    $(c)\Rightarrow(b)$ The sets $K_1:=\pi_1\cC$ and $K_2:=\pi_2\cC$ are compact and clearly the inclusion $\cC\subseteq\pi_1\cC\tm\pi_2\cC$ holds. 
%
%    $(b)\Rightarrow(a)$ Obviously, $\cC\subseteq K_1\tm K_2$ is compact as closed subset of the compact product $K_1\tm K_2$.
%\end{proof}

A mapping $G:X\to Y$ is called \emph{proper}, if for each compact subset $K\subset Y$ also the preimage $G^{-1}(K)\subseteq X$ is compact. 
\begin{lemma}
    \label{lemcharprop}
    For continuous mappings $G:X\to Y$ the following are equivalent:
    \begin{enumerate}
        \item[(a)] $G$ is \emph{proper on bounded, closed subsets} of $X$, i.e.\ for every bounded, closed subset $B\subset X$ the restriction $G|_B$ is proper, 

        \item[(b)] for every bounded sequence $(x_n)_{n\in\N}$ in $X$ such that $(G(x_n))_{n\in\N}$ converges in $Y$, the sequence $(x_n)_{n\in\N}$ has a convergent subsequence. 
    \end{enumerate}
\end{lemma}
\begin{proof} %Morris p11, Prop 1.4.8
    $(a)\Rightarrow(b)$ Let $(x_n)_{n\in\N}$ be a bounded sequence in $X$ and $B\subset X$ be a closed ball containing all $x_n$. Because the closure $\overline{\set{G(x_n)\mid n\in\N}}$ is compact in $Y$, its preimage $W$ in $B$ is compact. Due to $x_n\in W$ for all $n\in\N$ this implies (b).

    $(b)\Rightarrow(a)$ Let $B\subset X$ be bounded and closed, while $K\subset Y$ is supposed to be compact. We establish that $C:=G^{-1}(K)\cap B\subset X$ is compact. Thereto, if $(x_n)_{n\in\N}$ is a sequence in $X$, then there exists a convergent subsequence $(G(x_{k_n}))_{n\in\N}$ of $(G(x_n))_{n\in\N}$. Thanks to assumption (b) we obtain a further subsequence $(x_{k_{l_n}})_{n\in\N}$ being convergent in the subset $C\subset X$, which means that $\cC$ is compact by sequential compactness. 
\end{proof}

Properness of parameter dependent continuous mappings can be verified using the following elementary criterion: 
\begin{lemma}
    \label{lemproper}
    If a continuous mapping $G:X\tm\Lambda\to Y$ satisfies
    \begin{itemize}
        \item[(i)] $G(\cdot,\lambda):X\to Y$, $\lambda\in\Lambda$, is proper on every bounded, closed subset of $X$, 

        \item[(ii)] $\set{G(x,\cdot):\Lambda\to Y\mid x\in B}$ is equicontinuous for all bounded $B\subset X$, 
    \end{itemize}
    then $G$ is proper on every product $B\tm\Lambda_0$ of bounded, closed sets $B\subset X$ with compact sets $\Lambda_0\subseteq\Lambda$.
\end{lemma}
\begin{proof}
    Let $B\subset X$ be bounded, closed and $\Lambda_0\subseteq\Lambda$ be compact. Then $G$ is proper on $B\tm\Lambda_0$, if and only if for each compact $K\subseteq Y$ the intersection $G^{-1}(K)\cap(B\tm\Lambda_0)$ is compact. This, in turn, means that every sequence in $G^{-1}(K)\cap(B\tm\Lambda_0)$ has a convergent subsequence. Let $((x_n,\lambda_n))_{n\in\N}$ be such a sequence and because $K\subseteq Y$ is compact, there exists a subsequence $((x_{k_n^1},\lambda_{k_n^1}))_{n\in\N}$ such that the following limit
    \begin{equation}
        y:=\lim_{n\to\infty}G\bigl(x_{k_n^1},\lambda_{k_n^1}\bigr)
        \label{appstar}
    \end{equation}
    exists. Moreover, since $\Lambda_0$ is compact, there exists a convergent subsequence $(\lambda_{k_n^2})_{n\in\N}$ of $(\lambda_{k_n^1})_{n\in\N}$ having the limit $\lambda_0\in\Lambda_0$.

    Due to assumption (i), $G(\cdot,\lambda_0)$ is proper on bounded, closed subsets of $X$. This means, if we show $y=\lim_{k\to\infty}G(x_{k_n^2},\lambda_0)$, then there exists a $x_0\in B$ and another subsequence $(x_{k_n^3})_{n\in\N}$ with $\lim_{n\to\infty}x_{k_n^3}=x_0$. Hence, $\lim_{n\to\infty}G(x_{k_n^3},\lambda_{k_n^3})=(x_0,\lambda_0)$ concludes the proof. 

    Now, because assumption (ii) implies the limit relation
    $$
        \lim_{n\to\infty}d\bigl(G(x_{k_n^2},\lambda_0),G(x_{k_n^2},\lambda_{k_n^2})\bigr)=0
    $$
    we obtain from the triangle inequality that
    \begin{align*}
        0
        &\leq
        d(G(x_{k_n^3},\lambda_0),y)
        \leq
        d(G(x_{k_n^3},\lambda_{k_n^3}),y)+d(G(x_{k_n^3},\lambda_{k_n^3}),y)
        \xrightarrow[n\to\infty]{\eqref{appstar}}
        0
    \end{align*}
    holds, as desired. 
\end{proof}
\section{Parity and global bifurcations}
\label{appC}
Let $X,Y$ denote real Banach spaces and $U\subseteq X\tm\Lambda$ be nonempty, open and simply connected with an open interval $\Lambda\subseteq\R$. We essentially follow the modern approach of \cite{lopez:mora:04,lopez:sampedro:22,lopez:sampedro:24} and consider continuous mappings $G:U\to Y$ having the properties: 
\begin{itemize}
	\item[$(\mathbf{M_1})$] For every $\lambda\in\Lambda$ one has $(0,\lambda)\in U$ (cf.\ Fig.~\ref{fig1}) and
	$$
		G(0,\lambda)\equiv 0\on\Lambda, 
	$$

	\item[$(\mathbf{M_2})$] the partial derivative $D_1G:U\to L(X,Y)$ exists as a continuous function, 

	\item[$(\mathbf{M_3})$] $D_1G(x,\lambda)\in F_0(X,Y)$ for all $(x,\lambda)\in U$.
\end{itemize}
\begin{SCfigure}[2]
	\includegraphics[scale=0.5]{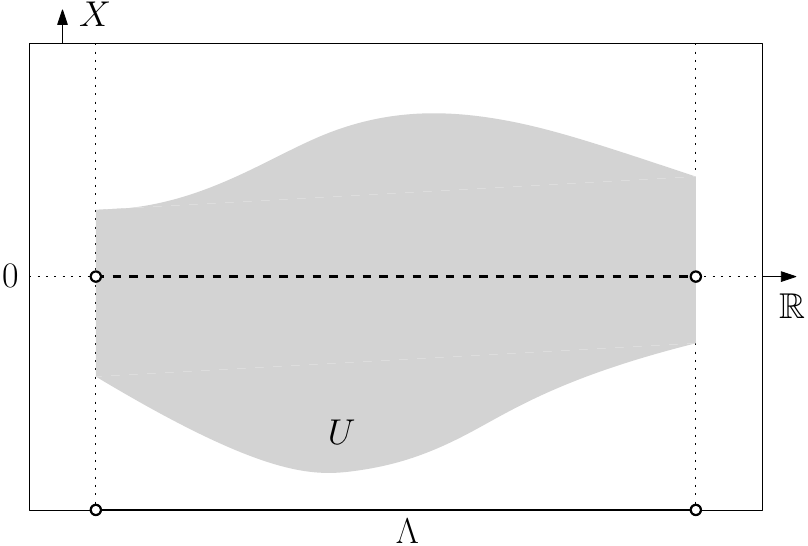}
	\caption{Role of the open interval $\Lambda\subseteq\R$ and shape of the set $U\subseteq X\tm\Lambda$ in Hypothesis $(M_1)$ with the trivial solution as dashed line}
	\label{fig1}
\end{SCfigure}

With the continuous path of index $0$ Fredholm operators
\begin{align*}
	T:\Lambda&\to F_0(X,Y),&
	T(\lambda)&:=D_1G(0,\lambda)
\end{align*}
we denote a parameter $\lambda\in\Lambda$ as \emph{critical}, if $T(\lambda)\not\in GL(X,Y)$ holds and introduce the \emph{set of critical values}
\begin{equation}
	\fC:=\set{\lambda\in\Lambda:\,T(\lambda)\not\in GL(X,Y)}.
	\label{nocv}
\end{equation}
Note that $\fC\subseteq\R$ is closed in $\Lambda$, but not necessarily discrete. Since each $T(\lambda)$ is Fredholm of index $0$, the set $\fC$ allows the characterization $\fC=\set{\lambda\in\Lambda:\,N(T(\lambda))\neq\set{0}}$. 

For any subinterval $[a,b]\subset\Lambda$ (of positive length) it is due to \cite{fitzPeja86} there exists a  \emph{parametrix} $P:[a,b]\to GL(Y,X)$ such that $P(\lambda)T(\lambda)-I_X\in L(X)$ is a compact operator for every $\lambda\in[a,b]$. In case $T: [a,b]\to F_0(X,Y)$ has invertible endpoints, then its \emph{parity} \cite{fitzPeja86,FiPejsachowiczV,FitPejRab,fitzpatrick:etal:92} is defined as
$$
	\sigma(T,[a,b])=\deg_{LS}(P(a)T(a))\cdot \deg_{LS}(P(b)T(b))\in\set{-1,1},
$$
where $\deg_{LS}$ stands for the Leray--Schauder degree (cf.\ e.g.\ \cite[pp.~199ff]{kielhoefer:12}). 

The parity is a tool to study the solution set $G^{-1}(0)\subseteq U$ of abstract parametrized equations
\begin{equation}
	\tag{$O_\lambda$}
	G(x,\lambda)=0.
\end{equation}
Since $G^{-1}(0)$ contains the trivial solutions in $\set{0}\tm\Lambda$ due to $(M_1)$, we introduce
$$
	\cS
	:=
	\set{(x,\lambda)\in U:\,G(x,\lambda)=0,\,x\neq 0}
	\cup
	(\set{0}\tm\fC)
$$
as \emph{set of nontrivial (or critical trivial) solutions} to \eqref{abs} (see Fig.~\ref{fig2}); it is closed in $U$. 

A pair $(0,\lambda^\ast)\in X\tm\Lambda$ is called \emph{bifurcation point}, if every neighborhood of $(0,\lambda^\ast)$ contains an element of $\cS$. In this case one denotes $\lambda^\ast$ as \emph{bifurcation value}; we abbreviate
$$
	\fB:=\set{\lambda\in\Lambda:\,(0,\lambda)\text{ is a bifurcation point for }\eqref{abs}}
$$
and obtain the inclusion $\fB\subseteq\fC$ from the Implicit Function Theorem \cite[p.~7, Thm.~I.1.1]{kielhoefer:12}. 
\begin{SCfigure}[2]
	\includegraphics[scale=0.5]{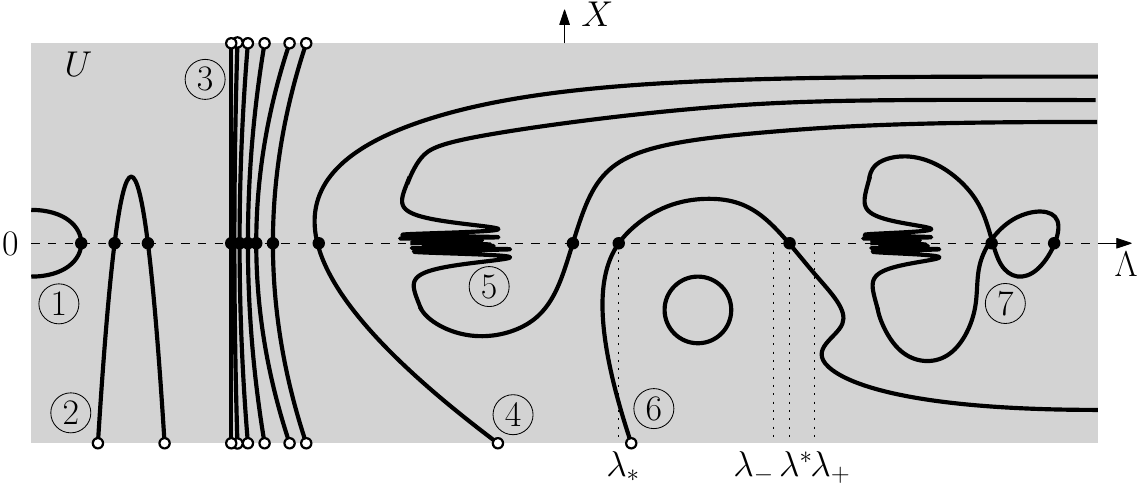}
	\caption{The set $\cS$ of nontrivial solutions and the alternatives from \tref{thmglobal} for $\Lambda=\R$ and $U$ bounded in $x$-direction (grey-shaded): 
	\newline
	1: $(a_2)$\newline
	2: $(a_1)$ \& (b)\newline
	3: $(a_1)$ \& $(a_2)$\newline
	4: $(a_1)$ \& $(a_2)$\newline
	5: $(a_2)$ \& (b) \newline
	6: $(a_1)$ \& $(a_2)$\& (b)\newline
	7: (b)
	}
	\label{fig2}
\end{SCfigure}
We denote a nonempty connected subset $\cC\subseteq\cS$ as \emph{component} of $\cS$, if it is maximal having these properties. With $\cS$ also every component $\cC$ of $\cS$ is closed in $U$. 
\begin{theorem}[abstract global bifurcation]
	\label{thmglobal}
	Let $(M_1$--$M_3)$ hold and $\lambda_-,\lambda_+\in\Lambda$ with $\lambda_-<\lambda_+$. 
	If $\sigma(T,[\lambda_-,\lambda_+])=-1$, then there exists a bifurcation value $\lambda^\ast\in(\lambda_-,\lambda_+)$. In particular, there exists a 
 component $\cC$ of $\cS$ intersecting $\set{0}\tm(\lambda_-,\lambda_+)$ in $(0,\lambda^\ast)$. Moreover, at least one of the following alternatives holds (cf.~Fig.~\ref{fig2}):
	\begin{enumerate}
		\item[(a)] The component $\cC$ is noncompact in $U$, i.e.\ its closure $\overline{\cC}$ in $X\tm\R$ does 
 %  \comment{\textcolor{red}{RS: I corrected a similar error as on page 24.}
  % \newline\textcolor{blue}{RS: Does Set 3 from in the above figure contain finitely many components or infinitely many? If infinitely, then $(a_2)$ in 3: is OK; otherwise, % if, for instance, $X$ is finite-dimensional space, then Set 3 is relatively compact and $(a_2)$ probably must be deleted.\newline RS: I think Set 3 in the above figure  % % satisfies $(a_1)$. Am I right?} }
		\begin{itemize}
			\item[$(a_1)$] contain a point $(x,\lambda)\in\partial U$, or 

			\item[$(a_2)$] is not a compact subset of $X\tm\R$, which in turn means that for any compact $X_0\subseteq X$ and compact $\Lambda_0\subseteq\R$ there exists a $(x,\lambda)\in\overline{\cC}\setminus(X_0\tm\Lambda_0)$, 
		\end{itemize}

		\item[(b)] $\cC$ contains a point $(0,\lambda_\ast)$ with $\lambda_\ast\in\fB\setminus[\lambda_-,\lambda_+]$. 
	\end{enumerate}
	In particular, for $U=X\tm\R$ and $\Lambda=\R$ the alternative $(a)$ reduces to $(a_2)$. 
\end{theorem}
\begin{proof}
	Since the domain $U$ is assumed to be simply connected, $G:U\to Y$ is orientable. In \cite[Prop.~5.6]{benevieri:furi:00} it is shown that $\sigma(T,[\lambda_-,\lambda_+])=\sgn T(\lambda_-)\sgn T(\lambda_+)$ holds, where also the notion of a sign for Fredholm operators $T(\lambda)$ is introduced. Hence, the signs $\sgn T(\lambda_-)$ and $\sgn T(\lambda_+)$ are different and \cite[Thm.~3.6]{benevieri:furi:01} implies the alternatives (a) or~(b). In particular, (a) is to be understood that $(a_1)$ or $(a_2)$ hold. Because $\overline{\cC}$ is closed, by to \lref{lemmaA}, $\overline{\cC}$ being compact would mean that this set is contained in the product of two compact subsets $X_0\subseteq X$, $\Lambda_0\subseteq\R$. However, the contraposition to this statement is simply $(a_2)$. Finally, the continuous differentiability of $G$ required to construct the parity $\sigma$ can be weakened to our assumption $(M_2)$ using methods due to Pejsachowicz (see \cite[Lemma~2.3.1]{pejsachowicz:11a} or \cite[Lemma~6.3]{Pej-Ski} together with \cite[Thm.~8.73]{vaeth:12}).
\end{proof}

\begin{theorem}[abstract global bifurcation for proper mappings $G$]
	\label{thmglobal2}
	Let $(M_1$--$M_3)$ hold and $\lambda_-,\lambda_+\in\Lambda$ with $\lambda_-<\lambda_+$. 
	If $\sigma(T,[\lambda_-,\lambda_+])=-1$ and $G:U\to Y$ is proper on closed and bounded subsets of $U$, then the alternative $(a_2)$ in \tref{thmglobal} simplifies to
	\begin{itemize}
		\item[$(a_2')$] $\cC$ is unbounded. 
	\end{itemize}
	In particular, for $U=X\tm\R$ and $\Lambda=\R$ the alternative $(a)$ reduces to $(a_2')$. 
\end{theorem}
\begin{proof}
	\tref{thmglobal2} traces \cite[Thm.~7.2]{fitzpatrick:etal:92}, where $G:U\to Y$ is assumed to be of class $C^2$ (however, see \cite[Rem.~7.2]{fitzpatrick:etal:92}). This deficit is avoided in \cite[Thm.~6.1]{pejsachowicz:rabier:98} requiring merely $C^1$-mappings $G$. Given this, \tref{thmglobal2} coincides with \cite[Thm.~6.1]{galdi:rabier:99}, whose smoothness assumptions can be weakened to our above setting based on references given in the proof of \tref{thmglobal}. 
\end{proof}

While it is clear that statement $(a_2')$ is nonlocal, indeed both the alternatives $(a)$ and $(b)$ are not of local nature. In particular, we argue that the component $\cC$ cannot be contained in arbitrarily small neighborhoods of $(0,\lambda^\ast)$ in $X\tm\R$: 
\begin{itemize}
	\item[ad (a)] Fredholm mappings are locally proper (cf.~\cite[(1.6)~Thm.]{smale:65}). Consequently, there exists a closed neighborhood $U\subseteq X\tm\R$ of $(0,\lambda^\ast)$ on which the restriction $G|_U$ is proper. Then the component $\cC$ is not contained in $U$ because otherwise it would be compact (due to $\cC\subseteq G^{-1}(0)$). 
 
    \item[ad (b)] On the one hand also $(0,\lambda_\ast)$ is a bifurcation point of \eqref{abs} and, in turn, $D_1G(0,\lambda_\ast)$ is singular due to the Implicit Function Theorem \cite[p.~7, Thm.~I.1.1]{kielhoefer:12}. On the other hand, as emphasized in \cite[p.~295]{galdi:rabier:99}, the possible point $\lambda_\ast$ closest to $\lambda^\ast$ is at positive distance from $\lambda^\ast$. Thus, because $\cC$ contains both $(0,\lambda^\ast)$ and $(0,\lambda_\ast)$, it cannot be confined to a neighborhood leaving out the closest point $(0,\lambda)$ such that $D_1G(0,\lambda)$ is singular.
\end{itemize}

For the sake of a refinement of alternative (b) in \tref{thmglobal}, we remind the reader of the open intervals $J_i$, $i\in\I$, satisfying $(I_1$--$I_3)$, and the compact intervals $\bar J_i$, $i\in\I'$, introduced in Sect.~\ref{sec5}. In particular, the family $J:=\set{\bar J_i}_{i\in\I'}$ was said to \emph{cover} the set $\fC$ of critical values abstractly defined in \eqref{nocv}. 
\begin{SCfigure}[2]
	\includegraphics[scale=0.5]{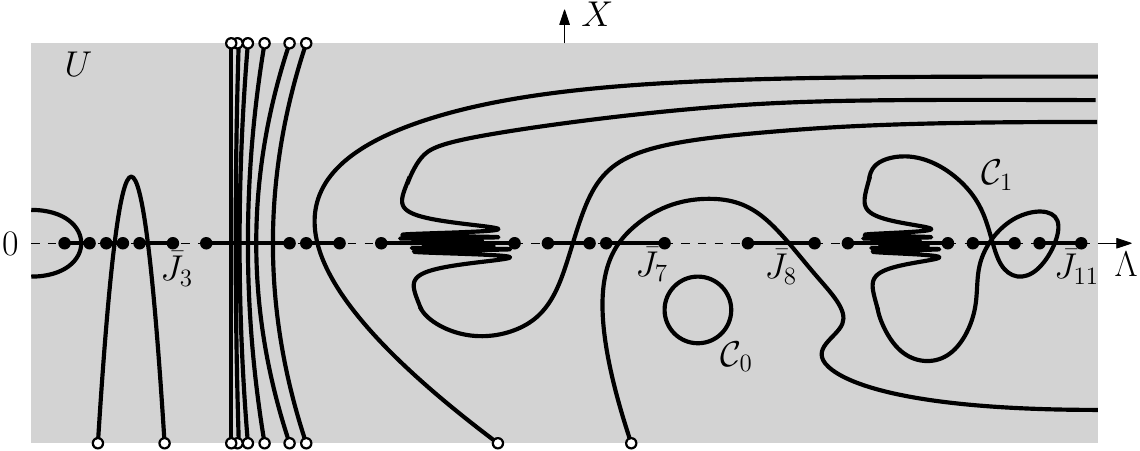}
	\caption{The set $\cS_J$ for $\Lambda=\R$ with $U$ grey-shaded. For the bounded component $\cC_0$ and $\cC_1$ of $\cS$ it is
	\newline $I_J(\cC_0)=0$ and
	\newline
	$I_J(\cC_1)=\set{9,10,11}$}
	\label{fig4}
\end{SCfigure}

On this basis, the superset of $\cS$ given by (cf.~Fig.~\ref{fig4}), 
$$
	\cS_J
	:=
	\set{(x,\lambda)\in U:\,G(x,\lambda)=0,\,x\neq 0}
	\cup
	\intoo{
	\set{0}\tm\bigcup_{i\in\I'}\bar J_i}
$$
is closed in $U$. The compact sets $\set{0}\tm\bar J_i$ connect  several components of $\cS$ bifurcating at bifurcation values in $\bar J_i$ to a single component of $\cS_J$ (cf.~Fig.~\ref{fig2} and~\ref{fig4}). 

Given a component $\cC$ of $\cS$, we introduce the set 
$$
	I_J(\cC):=\set{i\in\I':\,\cC\cap(\set{0}\tm \bar J_i)\neq\emptyset}
$$
of all indices $i$ for which solutions bifurcate from some $\set{0}\tm\bar J_i$ (cf.~Fig.~\ref{fig4}). Note that $I_J(\cC)$ is finite for compact $\cC$. For each $i\in\I$ we choose a $\lambda_i\in J_i$ and recursively define a sequence $(a_i)_{i\in\I}$ in $\set{-1,1}$ as follows: Fix $i_0\in\I$, either $a_{i_0}:=1$ or $a_{i_0}:=-1$, and set
\begin{align*}
	a_{i+1}&:=a_i\sigma(T,[\lambda_i,\lambda_{i+1}])\fall i\geq i_0,\,i\in\I',\\
	a_{i-1}&:=a_i\sigma(T,[\lambda_{i-1},\lambda_i])\fall i\leq i_0,\,i-1\in\I.
\end{align*}
Every $a_i$ indicates the orientation of $T(\lambda)\in GL(X,Y)$ in the interval $J_i$ (cf.~\cite{lopez:sampedro:22}). This allows us to introduce the $J$-\emph{parity map}
\begin{align*}
	\pi_J:\I'&\to\set{-1,0,1},&
	\pi_J(i)&:=\tfrac{a_{i+1}-a_i}{2},
\end{align*}
having the properties
\begin{align*}
	\pi_J(i)&=0\Leftrightarrow a_{i+1}=a_i\text{ i.e.\ $T$ has the same orientation on $J_i$ and $J_{i+1}$},\\
	\pi_J(i)&\neq 0\Leftrightarrow a_{i+1}\neq a_i\text{ i.e.\ $T$ has different orientation on $J_i$ and $J_{i+1}$}.
\end{align*}
It therefore measures the number of consecutive orientation changes of $T(\lambda)$ as the parameter $\lambda$ increases through the interval $\Lambda$ respecting the covering family $J$. 
\begin{theorem}[bounded components]
	\label{thmsign}
	Let $(M_1$--$M_3)$ hold, $G:U\to Y$ be proper on closed, bounded subsets of $U$ and suppose the family $J$ covers $\fC$. If $\cC\neq\emptyset$ is a bounded component of $\cS_J$, then $\cC$ is compact in $U$ and
	\begin{equation}
		\sum_{i\in I_J(\cC)}\pi_J(i)=0. 
		\label{thmsign1}
	\end{equation}
\end{theorem}
An interpretation of the condition \eqref{thmsign1} is as follows: For each index $i\in I_J(\cC)$ with $\pi_J(i)=\pm 1$ there is another $j\in I_J(\cC)$ with $\pi_J(j)=\mp 1$. In this case, the continuum $\cC$ connects the sets $\set{0}\tm\bar J_i$ and $\set{0}\tm\bar J_j$. In particular, there always exists an even number of indices $i\in I_J(\cC)$ with $\pi_J(i)\neq 0$. 

Furthermore, it is easy to see that the left-hand side in \eqref{thmsign1a} plays the role of the \emph{bifurcation index} used in the specifications \cite[p.~342, (1.9) and p.~344, (1.10)]{granas:dugundji:03} of the classical alternative from \cite{rabinowitz:71} or \cite[p.~205, Thm.~II.3.3]{kielhoefer:12} applying the the situation, where $T(\lambda)$ is a compact perturbation of the identity. 
\begin{proof}
	Because $G$ is proper on closed, bounded subsets, any bounded component $\cC\subset U$ is actually compact. Given this, the remaining argument follows \cite[Thm.~5.9]{lopez:sampedro:24}, where \tref{thmsign} is formulated for globally defined mappings $G$.
\end{proof}

\begin{corollary}
	\label{corsign}
	Let $U=X\tm\R$, $\Lambda=\R$, assume $\fC$ is a discrete set and let $\lambda_-,\lambda_+\in\R$ with $\lambda_-<\lambda_+$ and $\sigma(T,[\lambda_-,\lambda_+])=-1$. If $\cC$ denotes the continuum emanating from $\set{0}\tm(\lambda_-,\lambda_+)$ and one of the conditions
	\begin{itemize}
		\item[(i)] $\cC\cap(\set{0}\tm\fC)$ is infinite, 

		\item[(ii)] $\sum_{i\in I_J(\cC)}\pi_J(i)\neq 0$
	\end{itemize}
	holds, then $\cC$ is unbounded. 
\end{corollary}
\begin{proof}
	In case the intersection $\cC\cap(\set{0}\tm\fC)$ is an infinite set, then the discreteness of $\fC$ yields that the continuum $\cC$ is unbounded. 
	If $\cC\cap(\set{0}\tm\fC)$ is finite and bounded, then \eqref{thmsign1} holds and contradicts the assumption (ii). 
\end{proof}
\bibliographystyle{amsplain}

\end{document}